\documentclass[english]{amsart}
\usepackage{amsfonts, amsmath, amssymb, amscd, amsthm, array, graphicx}
\usepackage[latin9]{inputenc}
\usepackage{amsmath}
\usepackage{stmaryrd}
\usepackage{graphicx}
\usepackage{amssymb}
\usepackage{babel}
\usepackage{color}
\usepackage[all]{xy}

\usepackage{mathtools}
\usepackage{hyperref,url}

\makeatletter
\def\blfootnote{\xdef\@thefnmark{}\@footnotetext}
\makeatother
\usepackage{tikz}
\usetikzlibrary{arrows,shapes,snakes,automata,backgrounds,petri}
\usetikzlibrary{automata} 
\usetikzlibrary[automata] 
\usetikzlibrary{matrix,decorations.pathreplacing}
\usepackage{geometry}
\usepackage{pdflscape}
\usepackage{graphicx}
\usetikzlibrary{external,automata,trees,positioning,shadows,arrows,shapes.geometric}

\newlength\Textht
\newtheorem{thm}{Theorem}[section]
\newtheorem{cor}[thm]{Corollary}
\newtheorem{lemma}[thm]{Lemma}
\newtheorem{prop}[thm]{Proposition}

\newtheorem{que}[thm]{Question}

\theoremstyle{definition}

\newtheorem{rem}[thm]{Remark}
\newtheorem{ex}[thm]{Example}

\newcommand{\rk}{\operatorname{r}}
\newcommand{\id}{\operatorname{Id}}
\newcommand{\rrk}{\operatorname{\tilde{r}}}

\newcommand{\Z}{{\mathbb Z}}
\newcommand\ZZ{\mathbb{Z}}
\newcommand{\C}{{\mathbb C}}
\newcommand{\Q}{{\mathbb Q}}

\newcommand{\N}{{\mathbb N}}

\newcommand{\Mat}{\operatorname{M}}
\newcommand{\GL}{\operatorname{GL}}
\newcommand{\Aut}{\operatorname{Aut}}
\newcommand{\Out}{\operatorname{Out}}
\newcommand{\End}{\operatorname{End}}
\newcommand{\Inn}{\operatorname{Inn}}
\newcommand{\im}{\operatorname{Im}}
\newcommand{\Fix}{\operatorname{Fix}}
\newcommand{\acl}{\operatorname{a-Cl}}
\newcommand{\ecl}{\operatorname{e-Cl}}

\newcommand{\Per}{\operatorname{Per}}
\newcommand{\ord}{\operatorname{ord}}
\newcommand{\lcm}{\operatorname{lcm}}

\begin{document}

\title[Fixed subgroups and auto-fixed closures in free-abelian times free groups]{Fixed subgroups and computation of auto-fixed closures in free-abelian times free groups}

\author{Mallika Roy}
\address{Departament de Matem\`atiques, Universitat Polit\`ecnica de Catalunya, CATALONIA.} \email{mallika.roy@upc.edu}

\author{Enric Ventura}
\address{Departament de Matem\`atiques, Universitat Polit\`ecnica de Catalunya, CATALONIA.} \email{enric.ventura@upc.edu}



\begin{abstract}
The classical result by Dyer--Scott about fixed subgroups of finite order automorphisms of $F_n$ being free factors of $F_n$ is no longer true in $\Z^m\times F_n$. Within this more general context, we prove a relaxed version in the spirit of Bestvina--Handel Theorem: the rank of fixed subgroups of finite order automorphisms is uniformly bounded in terms of $m,n$. We also study periodic points of endomorphisms of $\Z^m\times F_n$, and give an algorithm to compute auto-fixed closures of finitely generated subgroups of $\Z^m\times F_n$. On the way, we prove the analog of Day's Theorem for real elements in $\Z^m\times F_n$, contributing a modest step into the project of doing so for any right angled Artin group (as McCool did with respect to Whitehead's Theorem in the free context).
\end{abstract}

\maketitle

\section{Introduction}

The goal of this paper is to investigate the properties of fixed point subgroups of automorphisms of direct products of free-abelian and free groups, $\Z^m \times F_n$. The lattice of subgroups of these groups is quite different from that of free groups, since $\Z^m\times F_n$ is not Howson (i.e., the intersection of two finitely generated subgroups is not necessarily finite generated) as soon as $m\geqslant 1$ and $n\geqslant 2$. This affects seriously to the behaviour of the rank function, forcing many situations to degenerate with respect to what happens in free groups. However, there are still several surviving governing rules; we concentrate on some of them, specially about those concerning subgroups fixed by automorphisms of $\Z^m \times F_n$.

Let $G$ be a group.

We denote by $\rk(G)$ the \emph{rank} of $G$, i.e., the minimal number of generators for $G$; also, $\rrk(G)=\max \{\rk(G)-1, 0\}$ denotes the \emph{reduced rank} of $G$. We denote by $\End(G)$ (resp., $\Aut(G)$) the monoid (resp., group) of endomorphisms (resp., automorphisms) of $G$, and write them all with the arguments on the left, $g\mapsto g\alpha$; so, accordingly, $\alpha\beta$ denotes the composition $g\mapsto g\alpha \mapsto g\alpha\beta$. Specifically, we will reserve the letter $\gamma$ for right conjugations, $\gamma_x\colon G\to G$, $g\mapsto x^{-1}gx$.

We will denote by $\Mat_{n\times m}(\Z)$ the $n\times m$ (additive) group of matrices over $\Z$, and by $\GL_m(\Z)$ the linear group over the integers. When thinking a matrix $A$ as a map, it will always act on the right of horizontal vectors, $v\mapsto vA$.

Given a set $S\subseteq \End(G)$, we let $\Fix(S)$ denote the subgroup of $G$ consisting of those $g\in G$ which are fixed by every element of $S$, $\Fix(S)=\{g\in G \mid g\alpha=g,\,\,\, \forall \alpha\in S\}=\cap_{\alpha\in S} \Fix(\{\alpha\})$, called the \emph{fixed subgroup} of $S$ (read $\Fix(\emptyset )=G$). For simplicity, we write $\Fix \phi =\Fix (\{\phi\})$.

For an endomorphism $\phi \in \End(G)$, define its \emph{periodic} subgroup as $\Per \psi =\cup_{p=1}^{\infty} \Fix \psi^p$ (note that this is always a subgroup since $x\in \Fix \psi^p$ and $y\in \Fix \psi^q$ imply $xy\in \Fix \psi^{pq}$). Observe that $\Per \psi$ contains the lattice of subgroups given by $\Fix \psi^p$, $p\in \N$, with inclusions among them exactly according to divisibility among the exponents: if $r|s$ then $\Fix \phi^r\leqslant \Fix \phi^s$; and also, if $\Fix \phi^r\leqslant \Fix \phi^s$ and $d=\gcd(r,s)=\alpha r+\beta s$, $\alpha,\beta \in \Z$, then $\Fix \phi^r=\Fix \phi^d$ and $d|s$.

Any direct product of a free-abelian group, $\Z^m$, $m\geqslant 0$, and a free group, $F_n$, $n\geqslant 0$, will be called, for short, a \emph{free-abelian times free} group, $G=\Z^m\times F_n$. We will work in $G$ with multiplicative notation (as it is a non-abelian group as soon as $n\geqslant 2$) but want to refer to its subgroup $\Z^m\leqslant G$ with the standard additive notation (elements thought as row vectors with addition). To make these compatible, consider the standard presentations $\Z^m=\langle t_1, \ldots ,t_m \mid [t_i,t_j],\,\,\, i,j=1,\ldots ,m\rangle$ and $F_n=\langle z_1, \ldots ,z_n \mid \,\,\rangle$, and the standard normal form for elements from $G$ with vectors on the left, namely $t_1^{\alpha_1}\cdots t_m^{\alpha_m}w(z_1, \ldots ,z_n)$, where $\alpha_1,\ldots ,\alpha_m\in \Z$ and $w\in F_n$ is a reduced word on the alphabet $Z=\{z_1, \ldots ,z_n\}$; then, let us abbreviate this in the form
 $$
t_1^{\alpha_1}\cdots t_m^{\alpha_m}w(z_1, \ldots ,z_n) =t^{(\alpha_1, \ldots ,\alpha_m)}w(z_1,\ldots ,z_n)=t^a w(z_1,\ldots ,z_n),
 $$
where $a=(\alpha_1,\ldots ,\alpha_m)\in \Z^m$ is the row vector made with the integers $\alpha_i$'s, and $t$ is a meaningless symbol serving only as a pillar for holding the vector $a=(\alpha_1,\ldots ,\alpha_m)$ up in the exponent. This way, the operation in $G$ is given by $(t^a u)(t^b v)=t^a t^b uv=t^{a+b}uv$ in multiplicative notation, while the abelian part works additively, as usual, up in the exponent. We denote by $\pi$ the natural projection to the free part, $\pi\colon \Z^m\times F_n \twoheadrightarrow F_n$, $t^a u\mapsto u$.

According to Delgado--Ventura~\cite[Def.~1.3]{DV}, a \emph{basis} of a finitely generated subgroup $H\leqslant_{fg} G$ is a set of generators for $H$ of the form $\{t^{a_1}u_1, \ldots ,t^{a_r}u_r,\, t^{b_1},\ldots ,t^{b_s}\}$, where $a_1, \ldots ,a_r\in \Z^m$, $\{ u_1, \ldots ,u_r\}$ is a free-basis of $H\pi\leqslant F_n$, and $\{ b_1,\ldots ,b_s\}$ is an abelian basis of $L_H =L\cap \Z^m \leqslant \Z^m$. (Note that, to avoid confusions, we reserve the word \emph{basis} for $G$, in contrast with \emph{abelian-basis} and \emph{free-basis} for the corresponding concepts in $\Z^m$ and $F_n$, respectively.) It was showed in~\cite{DV} that every such subgroup $H\leqslant_{fg} G$ admits a basis, algorithmically computable from any given set of generators. Furthermore, any subgroup $H\leqslant \Z^m\times F_n$, $n\geqslant 2$, is again free-abelian times free, $H\simeq \Z^{m'}\times F_{n'}$, for some $0\leqslant m'\leqslant m$ and some $0\leqslant n'\leqslant \infty$ (and hence, it is finitely generated if and only if $H\pi\leqslant F_n$ is so).

We recall from Delgado--Ventura~\cite[Props.~5.1,~5.2(iii)]{DV} that every automorphism $\Psi$ of the group $G=\Z^m \times F_n$, $n\geqslant 2$, is of the form $\Psi =\Psi_{\phi,Q,P}\colon G\to G$, $t^{a}u\mapsto t^{aQ+u^{\rm ab}P}(u\phi)$, where $\phi \in \Aut(F_n)$, $Q\in \GL_m(\Z)$, $P\in M_{n\times m}(\Z)$, and $u^{\rm ab}\in \Z^n$ is the abelianization of $u\in F_n$. Furthermore, the composition and inversion of automorphisms work like this:
 \begin{equation}\label{.}
\Psi_{\phi,Q,P} \Psi_{\phi',Q',P'}= \Psi_{\phi\phi',QQ',PQ'+AP'},\qquad (\Psi_{\phi,Q,P})^{-1}= \Psi_{\phi^{-1},Q^{-1},-A^{-1}PQ^{-1}},
 \end{equation}
where $A\in M_n(\Z)$ is the matrix of the abelianization of $\phi$; see~\cite[Lem.~5.4]{DV}. We shall use lowercase Greek letters for endomorphisms of free groups, $\phi \colon F_n \mapsto F_n$ and uppercase Greek letters for endomorphisms of free-abelian times free groups, $\Psi \colon \Z^m \times F_n \mapsto \Z^m \times F_n$. In particular, $\Gamma_{t^au}=\Gamma_u=\Psi_{\gamma_u, I_m, 0}\in \Inn(G)$ is the right conjugation by $t^au$ (or, equivalently, by $u$).

The paper is organized as follows. In Section~\ref{2}, we collect several folklore facts about $\GL_m(\Z)$ for later use; for completeness, we provide proofs highlighting several technical subtleties coming from the fact that $\Z$ is not a field, but just an integral domain. In Section~\ref{3}, we concentrate on finite order automorphisms of $\Z^m \times F_n$ and show that their fixed subgroups are always finitely generated, with rank globally bounded by a computable constant depending only on the ambient ranks $m,n$ (and not depending on the specific automorphism in use); see Theorem~\ref{fo fg}. In Section~\ref{4}, we turn to study periodic points and we manage to extend to free-abelian times free groups a result known to hold both in free-abelian groups and in free groups: the periodic subgroup of an endomorphism equals the fixed subgroup of a high enough power and, furthermore, this exponent can be taken uniform for all endomorphisms, depending only on the ambient ranks $m, n$; see Theorem~\ref{periodic-fatf}. In Section~\ref{5}, we consider the auto-fixed closure of a finitely generated subgroup $H$ (roughly speaking, the set of elements fixed by every automorphism fixing $H$); we prove that it always equals a finite intersection of fixed subgroups, we compute the candidate automorphisms, we decide whether it is finitely generated or not, and in case it is, we effectively compute a basis for it; see Theorem~\ref{closure}. As a consequence, we obtain an algorithm to decide whether a given finitely generated subgroup $H$ is auto-fixed or not; see Corollary~\ref{deciding}. To achieve this goal, we make use of a recent result by M. Day about stabilizers of tuples of conjugacy classes in right angled Artin groups being finitely presented, and we prove the analogous version for tuples of exact elements in $\Z^m\times F_n$. In fact, we only need finite generation and computability of these stabilizers; however, for completeness, we also prove its finite presentability postponing the analysis of the relations (a bit more technical) to the Appendix~\ref{6}.

\section{Preliminaries on $\GL_m(\Z)$}\label{2}

In this section we collect well known and folklore results about the general linear group over the integers, $\GL_m(\Z)$. This group is very well studied in the literature, but we are interested in highlighting several subtleties coming from the fact that $\Z$ is not a field, but just an integral domain.

\begin{lemma}\label{directsum}
Let $Q\in \GL_m(\Z)$ be a matrix such that $Q^k=I_m$. Then, we have the decomposition $\Z^m= \ker (Q-I_m) \oplus \ker (Q^{k-1}+ \cdots +Q+I_m)$.
\end{lemma}

\begin{proof}
Since $\gcd (x^{k-1}+\cdots +x+1,\, x-1)=1$, Bezout's equality gives us two polynomials $\alpha(x), \beta(x) \in \Z[x]$ such that $1=\alpha(x)(x^{k-1}+\cdots +x+1)+\beta(x)(x-1)$. Plugging $Q$, we obtain the matrix equality $I_m =\alpha(Q)(Q^{k-1}+\cdots +Q+I_m)+\beta(Q)(Q-I_m)$. Now, for every vector $v\in \Z^m$, we have $v=v\alpha(Q)(Q^{k-1}+\cdots +Q+I_m)+v\beta(Q)(Q-I_m)$. And, since $(Q-I_m)(Q^{k-1}+ \cdots +Q+I_m)=(Q^{k-1}+ \cdots +Q+I_m)(Q-I_m)=Q^k-I_m=0$, the first summand is in $\ker(Q-I_m)$ and the second one in $\ker (Q^{k-1}+ \cdots +Q+I_m)$; hence, $\Z^m= \ker (Q-I_m)+\ker (Q^{k-1}+ \cdots +Q+I_m)$.

Now let $v \in \ker (Q-I_m)\cap \ker (Q^{k-1}+\cdots +Q+I_m)$. This means that $v(Q-I_m)=0$ and $v(Q^{k-1}+ \cdots +Q+I_m)=0$, which imply $v=v(Q^{k-1}+\cdots +Q+I_m)\alpha(Q)+v(Q-I_m)\beta(Q)=0$. Thus, $\Z^m=\ker (Q-I_m) \oplus \ker (Q^{k-1}+\cdots +Q+I_m)$.
\end{proof}

\begin{prop}\label{Q-finite-order}
Consider the integral linear group $\GL_m(\Z)$, $m\geqslant 1$.
\begin{itemize}
\item[(i)] There exists a computable constant $L_1=L_1(m)$ such that, for every matrix $Q\in \GL_m(\Z)$ of finite order, $\ord(Q)\leqslant L_1$.
\item[(ii)] There exists a computable constant $L_2=L_2(m)$ such that, for every matrix $Q\in \GL_m(\Z)$ of finite order, say $k=\ord(Q)\leqslant L_1$, we have that $M=\im(Q-I_m)$ is a finite index subgroup of $\ker (Q^{k-1}+\cdots +Q+I_m)$ with $[\ker (Q^{k-1}+\cdots +Q+I_m) : M]\leqslant L_2$.
\end{itemize}
\end{prop}

\begin{proof}
(i) is a well known fact about integral matrices; we offer here a self-contained proof mixed with that of (ii).

Let $Q\in \GL_m(\Z)$ be a matrix of order $k<\infty$ (i.e., $Q^k=I_m$ but $Q^i\neq I_m$ for $i=1,\ldots ,k-1$).

Since $(Q-I_m)(Q^{k-1}+ \cdots +Q+I_m)=Q^k-I_m=0$, we have $M=\im (Q-I_m)\leqslant \ker (Q^{k-1}+\cdots +Q+I_m)$. But, by Lemma~\ref{directsum} and the Rank-Nullity Theorem, $\rk (M)= \rk (\im (Q-I_m))= m-\rk (\ker (Q-I_m))=\rk (\ker (Q^{k-1}+\cdots +Q+I_m))$ and so, $M\leqslant_{fi} \ker (Q^{k-1}+\cdots +Q+I_m)$. This is the index we have to bound globally in terms of $m$.

Let $m_Q(x)$ be the minimal polynomial of $Q$. Since $Q^k=I_m$, we have $m_Q(x) \,|\, x^k-1$ and so, $m_Q(x)=(x-\alpha_1)\cdots (x-\alpha_r)$, where $\alpha_1\, \ldots ,\alpha_r$ are pairwise different $k$-th roots of unity (in particular, all roots of $m_Q(x)$ are simple and so $Q$ diagonalizes over the complex field $\C$). Write $d_i =\ord(\alpha_i)$. Since cyclotomic polynomials $\Phi_{d_i}(x)$ are irreducible over $\Z$, we deduce $\Phi_{d_i}(x) \, \vert \, m_Q(x)$ and so, $\varphi(d_i)=\deg(\Phi_{d_i}(x))\leqslant \deg(m_Q(x))\leqslant m$, where $\varphi$ is the Euler $\varphi$-function. But it is well known that $\lim_{n\to \infty} \varphi(n)=\infty$; see, for example, Dummit--Foote~\cite[p.~8]{Dummit-Foote} from where we can compute a big enough constant $C=C(m)$ such that $d_1,\ldots ,d_r\leqslant C$. Finally, $k=\ord(Q)=\lcm(\ord(\alpha_1)\, \ldots ,\ord(\alpha_r))=\lcm(d_1, \ldots ,d_r)\leqslant d_1\cdots d_r\leqslant C^r\leqslant C^m$; this is the constant we are looking for in (i), $L_1=C(m)^m$.

On the other hand, diagonalyzing $Q$, we get an invertible complex matrix $P\in \GL_m(\C)$ such that $P^{-1}QP=D=\operatorname{diag}(\alpha_1, \stackrel{s_1}{\ldots} , \alpha_1, \ldots , \alpha_r,\stackrel{s_r}{\ldots}, \alpha_r)$, where $s_1,\ldots ,s_r$ are the multiplicities in the characteristic polynomial, $\chi_Q(x)=(x-\alpha_1)^{s_1} \cdots (x-\alpha_r)^{s_r}$. Since $\alpha_i$ is a primitive $d_i$-th root of unity, it can take $\varphi(d_i)\leqslant m$ many values and, since $s_1+\cdots +s_r=m$, the diagonal matrix $D$ can take only finitely many values; we can make a list of all of them (up to reordering of the $\alpha_i$'s) and, for each one, compute the index $[\ker (D^{k-1}+\cdots +D+I_m) : \im (D-I_m)]$. The maximum of these indices is the constant $L_2=L_2(m)$ we are looking for in (ii), because
 $$
[\ker (Q^{k-1}+\cdots +Q+I_m) : M]=[(\ker (Q^{k-1}+\cdots +Q+I_m))P : (\im (Q-I_m))P]=
 $$
 $$
=[\ker P^{-1}(Q^{k-1}+\cdots +Q+I_m)P : \im (P^{-1}(Q-I_m)P)]=[\ker (D^{k-1}+\cdots +D+I_m) : \im (D-I_m)].
 $$
\end{proof}

We study now the \emph{periodic} subgroup of a matrix $Q\in \Mat_m(\Z)$, namely $\Per Q=\{v\in \Z^m \mid vQ^p =v, \text{ for some } p\geqslant 1\}$. The next Proposition states that a uniform single exponent depending only on $m$, $L_3=L_3(m)$, is enough to capture \emph{all} the periodicity of \emph{all} $m\times m$ matrices $Q$.

\begin{prop}\label{periodic-abelian}
There exists a computable constant $L_3=L_3(m)$ such that $\Per Q=\Fix Q^{L_3}$, for every $Q\in \Mat_m(\Z)$.
\end{prop}

\begin{proof}
As we argued in the proof of Proposition~\ref{Q-finite-order}(i), there is a computable constant $C=C(m)$ such that $\varphi(d)>m$ for every $d>C(m)$; see Dummit--Foote~\cite[p.~8]{Dummit-Foote}. Let us prove that the statement is true with the constant $L_3=C(m)!$

Fix a matrix $Q\in \Mat_m(\Z)$, and consider its characteristic polynomial factorized over the complex field $\C$, $\chi_Q(x)=(x-\alpha_1)^{s_1} \cdots (x-\alpha_r)^{s_r}$, where $\alpha_i\neq \alpha_j$, $i\neq j$. Standard linear algebra tells us that $\C^m=K_{\alpha_1} \oplus \cdots \oplus K_{\alpha_r}$, where $K_{\alpha_i}=\ker (Q-\alpha_i I_m)^{s_i}\leqslant \C^m$ is the generalized eigenspace of $Q$ with respect to $\alpha_i$, a $Q$-invariant $\C$-subspace of $\C^m$. Distinguish now between those $\alpha_i$'s which are roots of unity, say $\alpha_1,\ldots ,\alpha_{r'}$, and those which are not, say $\alpha_{r'+1}, \ldots ,\alpha_r$, $0\leqslant r'\leqslant r$. Write $d_i=\ord(\alpha_i)$, for $i=1,\ldots ,r'$, and observe that $d_1,\ldots ,d_{r'}\leqslant C$ (since the cyclotomic polynomials $\Phi_{d_i}(x)$ are $\Q$-irreducible and so must divide $\chi_Q(x)\in \Z[X]$, which has degree $m$); in particular, $\alpha_i^{L_3}=1$, $i=1,\ldots ,r'$.

Now, let $v\in \Per Q$, i.e., $vQ^p=v$ for some $p\geqslant 1$. Applying the above decomposition, $v=v_1+\cdots +v_r$, where $v_i \in K_{\alpha_i}$, and the $Q$-invariance of $K_{\alpha_i}$, we get the alternative decomposition $v=vQ^p=v_1Q^p +\cdots +v_rQ^p$. So, $v_iQ^p=v_i$, i.e., $v_i(Q^p-I_m)=0$, for $i=1,\ldots ,r$. For a fixed $i$, distinguish the following two cases:
\begin{itemize}
\item[(i)] if $\alpha^p_i\neq 1$, then $\alpha_i$ is not a root of $x^p-1$ and so, $1=\gcd \big( (x-\alpha_i)^{s_i},\, x^p-1\big)$. By Bezout's equality, there are polynomials $a(x), b(x) \in \C[x]$ such that $1=(x-\alpha_i)^{s_i} a(x) +(x^p-1)b(x)$. Plugging the matrix $Q$ and multiplying by the vector $v_i$ on the left, we obtain $v_i=v_i(Q-\alpha_i I_m)^{s_i} a(Q)+v_i(Q^p-I_m)b(Q)=0$.
\item[(ii)] if $\alpha^p_i =1$, then $x-\alpha_i =\gcd \big( (x-\alpha_i)^{s_i}, x^p-1\big)$. By Bezout's equality, there are polynomials $a(x), b(x) \in \C[x]$ such that $x-\alpha_i =(x-\alpha_i)^{s_i} a(x) +(x^p-1)b(x)$. Now, plugging the matrix $Q$ and multiplying by the vector $v_i$ on the left, we have $v_i(Q-\alpha_i I_m)=v_i(Q-\alpha_i I_m)^{s_i} a(Q)+v_i(Q^p-I_m)b(Q)=0$. That is, $v_iQ=\alpha_i v_i$ and so, $v_iQ^{L_3}=\alpha_i^{L_3} v_i =v_i$.
\end{itemize}

Altogether, $v=v_1+\cdots +v_r =\sum_{i\,|\, \alpha_i^p =1} v_i$ and $vQ^{L_3}=\big( \sum_{i\,|\, \alpha_i^p =1} v_i \big) Q^{L_3}=\sum_{i\,|\, \alpha_i^p =1} v_i Q^{L_3} =\sum_{i\,|\, \alpha_i^p =1} v_i =v$, and $v\in \Fix Q^{L_3}$. This completes the proof that $\Per Q=\Fix Q^{L_3}$.
\end{proof}

\section{Finite order automorphisms of $\Z^m\times F_n$}\label{3}

A well-known (and deep) result by Bestvina--Handel~\cite{BH} establishes a uniform bound (in fact, the best possible) for the rank of the fixed subgroup of any automorphism of $F_n$: for every $\phi\in \Aut(F_n)$, $\rk(\Fix \phi)\leqslant n$. This result followed an interesting previously know particular case due to Dyer--Scott~\cite{DS}: if $\phi\in \Aut(F_n)$ is of finite order then $\Fix\phi$ is a free factor of $F_n$.

When we move to a free-abelian times free group, $G=\Z^m \times F_n$, the situation degenerates, but still preserving some structure. In Delgado--Ventura~\cite{DV}, the authors gave an example of an automorphism $\Psi\in \Aut(G)$ with $\Fix\Psi$ \emph{not} being finitely generated; so, there is no possible version of Bestvina--Handel result in $G$. Following the parallelism, we show below an example of an automorphism $\Psi \in \Aut(G)$ of finite order (in fact, of order 2) such that $\Fix \Psi$ is \emph{not} a factor of $G$; see Example~\ref{ex-order-2}. However, as a positive result, in Theorem~\ref{fo fg}(ii) below we prove that finite order automorphisms of $G$ do have finitely generated fixed subgroups, in fact with a computable uniform upper bound for its rank, in terms of $m$ and $n$.

\begin{lemma}\label{ff}
Let $G=\Z^m \times F_n$. For given finitely generated subgroups $H\leqslant_{fg} K\leqslant_{fg} G$, the following are equivalent:
\begin{itemize}
\item[(a)] every basis of $H$ extends to a basis of $K$;
\item[(b)] some basis of $H$ extends to a basis of $K$;
\item[(c)] $H\pi \leqslant_{ff} K\pi$ and $L_H\leqslant_{\oplus} L_K$.
\end{itemize}
\end{lemma}

In this case, we say that $H$ is a \emph{factor} of $K$, denoted $H\leqslant_{f}K$; this is the notion in $G$ corresponding to \emph{free factor} in $F_n$ (denoted $\leqslant_{ff}$), and \emph{direct summand} in $\Z^m$ (denoted $\leqslant_{\oplus}$).

\begin{proof}
$(a)\,\Rightarrow \,(b)$ is obvious.

Assuming (b), we have $H=\langle t^{a_1}u_1, \ldots ,t^{a_r}u_r,\, t^{b_1},\ldots ,t^{b_s}\rangle$ and $K=\langle t^{a_1}u_1, \ldots ,t^{a_r}u_r, t^{a_{r+1}}u_{r+1},$ $\ldots ,t^{a_{r+p}}u_{r+p},\, t^{b_1},\ldots ,t^{b_s}, t^{b_{s+1}}, \ldots ,t^{b_{s+q}}\rangle$, where $\{u_1, \ldots ,u_r\}$ is a free-basis of $H\pi$, $\{b_1, \ldots ,b_s\}$ is an abelian-basis of $L_H$, $\{u_1, \ldots ,u_{r+p}\}$ is a free-basis of $K\pi$, and $\{b_1, \ldots ,b_{s+q}\}$ is an abelian-basis of $L_K$. Therefore, $H\pi \leqslant_{ff} K\pi$ and $L_H\leqslant_{\oplus} L_K$. This proves $(b)\,\Rightarrow \,(c)$.

Finally, assume (c). Given any basis $\{t^{a_1}u_1, \ldots ,t^{a_r}u_r,\, t^{b_1},\ldots ,t^{b_s}\}$ for $H$, $\{u_1, \ldots ,u_r\}$ is a free-basis of $H\pi$ (which can be extended to a free-basis $\{u_1, \ldots ,u_r, u_{r+1},\ldots ,u_{r+p}\}$ of $K\pi$ since $H\pi\leqslant_{ff}K\pi$); and  $\{b_1, \ldots ,b_s\}$ is an abelian-basis of $L_H$ (which can be extended to an abelian-basis $\{b_1, \ldots ,b_s, b_{s+1}, \ldots ,b_{s+q}\}$ of $L_K$ since $L_H\leqslant_{\oplus} L_K$). Then, choose vectors $a_{r+1}, \ldots ,a_{r+p}\in \Z^m$ such that $t^{a_{r+1}}u_{r+1}, \ldots ,t^{a_{r+p}}u_{r+p}\in K$ (this is always possible because $u_{r+1}, \ldots ,u_{r+p}\in K\pi$), and $\{ t^{a_1}u_1, \ldots ,t^{a_r}u_r, t^{a_{r+1}}u_{r+1},\ldots ,t^{a_{r+p}}u_{r+p},\, t^{b_1},\ldots ,t^{b_s}, t^{b_{s+1}}, \ldots ,t^{b_{s+q}}\}$ is a basis of $K$ (in fact, they generate $K$, and have the appropriate form). This proves $(c)\, \Rightarrow\, (a)$.
\end{proof}

\begin{thm}\label{fo fg}
Let $G=\Z^m \times F_n$, $m, n\geqslant 0$.
\begin{itemize}
\item[(i)] There exists a computable constant $C_1=C_1(m,n)$ such that, for every $\Psi\in \Aut(G)$ of finite order, $\ord(\Psi)\leqslant C_1$.
\item[(ii)] There exists a computable constant $C_2=C_2(m,n)$ such that, for every $\Psi\in \Aut(G)$ of finite order, $\rk(\Fix \Psi)\leqslant C_2$.
\end{itemize}
\end{thm}

\begin{proof}
(i). By Proposition~\ref{Q-finite-order}(i), the set $\{\ord(Q) \mid Q\in \GL_m(\Z) \mbox{ of finite order}\}$ is bounded above by a computable constant $L_1(m)$. And by Lyndon--Schupp~\cite[Cor.~I.4.15]{LS}, $\{\ord(\phi) \mid \phi\in \Aut(F_n)\mbox{ of finite order}\} \subseteq \{\ord(Q) \mid Q\in \GL_n(\Z) \mbox{ of finite order}\}$, which is bounded above by $L_1(n)$.

If $n\leqslant 1$ then $G=\Z^{m+n}$ is free-abelian and the constant $C_1=L_1(m+n)$ makes the job; if $m=0$ then $G=F_n$ is free and the constant $C_1=L_1(n)$ makes the job.

So, suppose $m\geqslant 1$, $n\geqslant 2$, and take an automorphism $\Psi=\Psi_{\phi,Q,P}\in \Aut(G)$. By Delgado--Ventura~\cite[Lemma~5.4(ii)]{DV}, $\Psi_{\phi,Q,P}^k=\Psi_{\phi^k, Q^k, P_k}$, where $P_k=\sum_{i=0}^{k-1} A^iPQ^{k-1-i}$ and $A\in \GL_n(\Z)$ is the abelianization of $\phi$. In particular, if $\Psi$ is of finite order then $\phi$ and $Q$ are so too; furthermore, $\ord(\Psi)=\lambda r_3$, where $r_3=\lcm (r_1, r_2)$, $r_1=\ord(\phi)$, and $r_2=\ord(Q)$. But $\Psi^{r_3}=\Psi_{id, id, P_{r_3}}$ and $\Psi^{\lambda r_3}=(\Psi_{id, id, P_{r_3}})^{\lambda}=\Psi_{id, id, \lambda P_{r_3}}$. Hence, $\Psi$ is either of order $r_3$ or of infinite order. In other words, $\{\ord(\Psi) \mid \Psi\in \Aut(G) \mbox{ of finite order}\}\subseteq \{\lcm (\ord(\phi),\, \ord(Q)) \mid \phi\in \Aut(F_n),\, Q\in \GL_m(\Z), \mbox{ both of finite order}\}$, which is bounded above by the constant $C_1(m,n)=L_1(n)L_1(m)$.

(ii). If $n\leqslant 1$ then $C_2=m+n$ makes the job, if $m=0$ then $C_2=n$ makes the job.

So, suppose $m\geqslant 1$, $n\geqslant 2$. Delgado--Ventura~\cite[\S 6]{DV} discusses the form of the fixed subgroup of a general automorphism $\Psi_{\phi,Q,P} \in \Aut(G)$, namely, $L_{\Fix \Psi}=\Fix(Q)=E_1(Q)$ (the eigenspace of eigenvalue 1 for $Q$), and $(\Fix \Psi)\pi =NP'^{-1}\rho'^{-1}$, where $\rho\colon F_n \twoheadrightarrow \Z^n$ is the abelianization map, $\rho'$ is its restriction to $\Fix \phi$, $P'$ is the restriction of $P$ to $\im \rho'$, $M=\im(Q-I_m)$, $N=M\cap \im P'$, and $(\Fix \Psi) \pi =N P'^{-1} \rho'^{-1}\unlhd \Fix \phi \leqslant F_n$, see the following diagram,

\begin{equation}\index{punts fixes per un automorfisme!diagrames}\label{15}
 \xy
 (65,0)*+{\geqslant M = \operatorname{Im} ({Q}-{I_m})};
 (58,-18.1)*+{ = M \cap \im {{P'}} .};
 (0,-4.5)*+{\rotatebox[origin=c]{90}{$\leqslant$}};
 (24,-4.5)*+{\rotatebox[origin=c]{90}{$\unlhd$}};
 (46,-4.5)*+{\rotatebox[origin=c]{90}{$\unlhd$}};
 {\ar@{->>}^-{\rho} (0,0)*++{F_{n}}; (24,0)*++{\ZZ^{n}}};
 {\ar^-{{P}} (24,0)*++++{}; (46,0)*++{\ZZ^m}};
 {\ar@{->>}^-{\rho'} (0,-9)*++{\Fix \phi}; (24,-9)*++{\im \rho'}};
 {\ar@{->>}^-{{P'}} (24,-9)*+++++{}; (46,-9)*++{\im {P'}}};
 (0,-13.5)*+{\rotatebox[origin=c]{90}{$\unlhd$}};
 (24,-13.5)*+{\rotatebox[origin=c]{90}{$\unlhd$}};
 (46,-13.5)*+{\rotatebox[origin=c]{90}{$\unlhd$}};
 {\ar@(ur,ul)_{Q-I_m} (46,0)*++{}; (46,0)*++{} };
 {\ar@{|->} (46,-18)*+++{N }; (24,-18)*++{NP'^{-1}}};
 {\ar@{|->} (24,-18)*+++++++{}; (0,-18)*+{NP'^{-1} \rho'^{-1} }};
 (-17,-18)*+{(\Fix \Psi) \pi =};
 \endxy
\end{equation}

If $\Fix \phi$ is trivial or cyclic, then $\rk (\Fix \Psi)=\rk((\Fix\Psi)\pi)+\rk(E_1(Q))\leqslant 1+m$. So, taking $C_2(m,n)\geqslant 1+m$, we are reduced to the case $\rk(\Fix \phi)\geqslant 2$.

With this assumption, $(\Fix \Psi)\pi \neq 1$ (it always contains the commutator of $\Fix\phi$) and so, $\Fix \Psi \leqslant G$ is finitely generated if and only if $(\Fix \Psi)\pi \leqslant F_n$ is so, which is if and only if the index $\ell:=[\Fix \phi : (\Fix \Psi)\pi ]=[\Fix \phi : NP'^{-1}\rho'^{-1}]=[\im \rho' : NP'^{-1}]=[\im P' : N]$ is finite. In this case, by the Schreier index formula, $\rrk(\Fix \Psi)=\rrk((\Fix\Psi)\pi) +\rk(E_1(Q))\leqslant \ell \rrk(\Fix \phi)+m\leqslant \ell(n-1)+m$. Therefore, we are reduced to bound the index $\ell$ in terms of $n$ and $m$.

First, let us prove that $\Psi$ being of finite order implies $\ell=[\im P' : N]<\infty$.

Put $k=\ord(\Psi_{\phi, Q, P})$ so, $\phi^k=\id$, $Q^k=I_m$, and $P_k =\sum_{i=0}^{k-1} A^iPQ^{k-1-i}=0$, where $A\in \GL_n(\Z)$ is the abelianization of $\phi$. By Proposition~\ref{Q-finite-order}(ii), the subgroup $M=\im(Q-I_m)$ is a finite index subgroup of $\ker (Q^{k-1}+\cdots +Q+I_m)$, with the index bounded above by a computable constant depending only on $m$, $[\ker (Q^{k-1}+\cdots +Q+I_m) : M]\leqslant L_2(m)$.

We claim that $\im P'\leqslant \ker (Q^{k-1}+\cdots +Q+I_m)$. In fact, take $u\in \Fix \phi$, note that $u\phi=u$ and so $(u\rho')A=u\phi \rho'=u\rho'$, and split $(u\rho')P' =v_1+v_2$, with $v_1 \in \ker (Q-I_m)$ and $v_2 \in \ker (Q^{k-1}+ \cdots +Q+I_m)$; see Lemma~\ref{directsum}. Multiplying by $Q^{k-1}+ \cdots +Q+I_m$ on the right,
 $$
v_1(Q^{k-1}+ \cdots +Q+I_m) = (v_1+v_2)(Q^{k-1}+ \cdots +Q+I_m) = (u\rho') P'(Q^{k-1}+ \cdots +Q+I_m) =
 $$
 $$
=\sum_{i=0}^{k-1} (u\rho')PQ^{k-1-i} = \sum_{i=0}^{k-1} (u\rho')A^iPQ^{k-1-i} = (u\rho') \sum_{i=0}^{k-1} A^iPQ^{k-1-i} = (u\rho')P_k=0,
 $$
from which we deduce $v_1\in \ker (Q-I_m) \cap \ker (Q^{k-1}+ \cdots +Q+I_m)=\{ 0\}$ so, $(u\rho')P'=v_2 \in \ker (Q^{k-1}+ \cdots +Q+I_m)$. Therefore, $\im P'\leqslant \ker (Q^{k-1}+ \cdots +Q+I_m)$.

Finally, intersecting the inclusion $M\leqslant_{fi} \ker (Q^{k-1}+ \cdots +Q+I_m)$ with $\im P'$, we get $N=M\cap \im P' \leqslant_{fi} \im P'$, and $\ell=[\im P' : N]\leqslant [\ker (Q^{k-1}+ \cdots +Q+I_m) : M]\leqslant L_2(m)$. Hence, taking $C_2(m,n)\geqslant L_2(m)(n-1)+m$ will suffice for the present case.

Therefore, $C_2(m,n)=L_2(m)(n-1)+m+1$ serves as the upper bound claimed in (ii).
\end{proof}

\begin{ex}\label{ex-order-2}
Here is an example of an order 2 automorphism of $G=\Z^2\times F_3$ whose fixed subgroup \emph{is not} a factor of $G$. Consider the automorphism $\Psi_{\phi, Q, P}$ determined by $\phi\colon F_3 \to F_3$, $z_1\mapsto z_1^{-1}$, $z_2\mapsto z_2$, $z_3\mapsto z_3$, $Q=\left(\begin{smallmatrix} 1 & 0 \\ 0 & -1 \end{smallmatrix}\right)\in \GL_2(\Z)$, and $P=\left( \begin{smallmatrix} 1 & 0 \\ 0 & 1 \\ 0 & 2 \end{smallmatrix} \right) \in M_{3\times 2}(\Z)$, i.e.,
 $$
\begin{array}{rcl}
\Psi \colon \Z^2\times F_3 & \longrightarrow & \Z^2\times F_3 \\
z_1 & \longmapsto & t^{(1,0)}z_1^{-1}\\
z_2 & \longmapsto & t^{(0,1)}z_2 \\
z_3 & \longmapsto & t^{(0,2)}z_3 \\
t^{(1,0)} & \longmapsto & t^{(1,0)} \\
t^{(0,1)} & \longmapsto & t^{(0,-1)}.
\end{array}
 $$
An easy computation shows that $\Psi^2=\id$, i.e., $\Psi$ has order 2. To compute $\Fix \Psi$, let us follow diagram~\eqref{15}: first note that $\Fix \phi = \langle z_2,z_3 \rangle$; so, $\im \rho' = \langle (0,1,0),(0,0,1) \rangle$, $\im P'=\langle (0,1),(0,2) \rangle=\langle (0,1) \rangle$. On the other hand, $M=\langle (0,2) \rangle$, $N=\langle (0,2)\rangle$, and $N{P'}^{-1}= \langle (0,2,0), (0,0,1) \rangle$. Therefore, $(\Fix \Psi)\pi =N{P'}^{-1}{\rho'}^{-1}=\{w(z_2,z_3) \mid |w|_{z_2} \text{ even}\}=\langle z_2^2,z_3, z_2^{-1}z_3z_2 \rangle$. So, solving the systems of equations to compute the vectors associated with each element of the free part, we obtain that $t^{(0,1)}z_2^2,\, t^{(0,1)}z_3,\, t^{(0,1)}z_2^{-1}z_3z_2 \in \Fix \Psi$. Finally, since $(\Fix \Psi)\cap \Z^2=E_1(Q)=\langle (1,0) \rangle$, we deduce that $\Fix \Psi = \langle t^{(0,1)}z_2^2,\, t^{(0,1)}z_3,\, t^{(0,1)}z_2^{-1}z_3z_2,\, t^{(1,0)} \rangle$.

Since $H\pi =\langle z_2^2, z_3, z_2^{-1}z_3z_2\rangle$ is not a free factor of $F_3$, $\Fix \Psi$ is not a \emph{factor} of $\Z^2\times F_3$; see Lemma~\ref{ff}.
\end{ex}

Theorem~\ref{fo fg} has the following easy corollary:

\begin{cor}\label{aa}
Let $\Psi \in \End(\Z^m \times F_n)$. If $\Fix \Psi^p$ is finitely generated then $\Fix \Psi$ is also finitely generated; the converse is not true.
\end{cor}

\begin{proof}
Clearly, $\Psi$ restricts to an automorphism $\Psi_| \in \Aut(\Fix \Psi^p)$ such that $\Fix \Psi_| =\Fix \Psi$ and $(\Psi_|)^p=\id$. Since $\Fix \Psi^p$ is finitely generated, we have $\Fix \Psi^p \simeq \Z^{m'}\times F_{n'}$ for some $m'\leqslant m$ and $n'<\infty$ and, applying Theorem~\ref{fo fg}(ii), we get $\rk(\Fix \Psi)=\rk(\Fix \Psi_|)<\infty$ (in fact, bounded above by $C_2(m',n')$).

The converse is not true as the following example shows. Consider $\Psi \colon \Z \times F_2 \to \Z \times F_2$, $z_1\mapsto tz_1^{-1}$, $z_2\mapsto z_2^{-1}$, $t\mapsto t^{-1}$. It is straightforward to see that $\Fix \Psi ={1}$. But $\Psi^2 \colon \Z \times F_2 \to \Z \times F_2$, $z_1\mapsto t^{-2}z_1$, $z_2\mapsto z_2$, $t\mapsto t$ and so, $\Fix \Psi^2 = \langle t \rangle \times \{ w(z_1,z_2) \in F_2 \mid |w|_{z_1} =0 \}=\langle t \rangle \times \langle \langle z_2 \rangle \rangle$ is not finitely generated.
\end{proof}

\section{Periodic points of endomorphisms of $\Z^m\times F_n$}\label{4}

Corollary~\ref{aa} states that, for $\Psi\in \Aut(G)$, the lattice of fixed subgroups of powers of $\Psi$ could simultaneously contain finitely and non-finitely generated subgroups but, as soon as one of them is finitely generated, the smaller ones must be so.

In the abelian case $G=\Z^m$, this lattice of fixed subgroups is always finite, and coming from a set of exponents uniformly bounded by $m$; this is precisely the contents of Proposition~\ref{periodic-abelian}. In the free case, combining results from Bestvina--Handel, Culler, Imrich--Turner, and Stallings, the exact analogous statement is true:

\begin{prop}[{Bestvina--Handel--Culler--Imrich--Turner--Stallings~\cite{BH, Cu, IT, St}; see also~\cite[Prop.~3.1]{BMMV}}]\label{periodic-free}
For every $\phi\in \End(F_n)$, we have $\Per \phi=\Fix \phi^{(6n-6)!}$.
\end{prop}

\begin{proof}
Culler~\cite{Cu} proved that every finite order element in $\Out (F_n)$ has order dividing $(6n-6)!$; and the same is true in $\Aut(F_n)$ since the natural map $\Aut(F_n)\twoheadrightarrow \Out(F_n)$ has torsion-free kernel. On the other hand Stallings~\cite{St} proved that, for every $\phi \in \Aut(F_n)$, there exists $s\geqslant 0$ such that $\Per \phi =\Fix \phi^s$. Also, Imrich--Turner~\cite{IT} proved that the so-called stable image of an endomorphism $\phi\in \End(F_n)$, namely $F\phi^{\infty}=\cap_{p=1}^{\infty} F_n\phi^p$, has rank at most $n$, it is $\phi$-invariant, it contains $\Per \phi$, and the restriction $\phi_|\colon F\phi^{\infty}\to F_n\phi^{\infty}$ is bijective. Finally, Bestvina--Handel Theorem (see~\cite{BH}) estates that $\rk(\Fix \phi)\leqslant n$, for any $\phi\in \Aut(F_n)$.

Combining these four results we can easily deduce the statement: given an endomorphism $\phi\colon F_n \to F_n$, consider its restrictions $\phi_1\colon F_n\phi^{\infty}\to F_n\phi^{\infty}$ and $\phi_2\colon \Per \phi_1 \to \Per \phi_1$, both bijective; furthermore, $\Per \phi_2 =\Per \phi_1 =\Fix \phi^s_1$ (assume $s\geqslant 0$ minimal possible), $\rk(\Per \phi_1)\leqslant \rk(F\phi^{\infty})\leqslant n$, and $\phi_2$ has order $s$. Therefore, $s$ divides $(6\rk(\Per\phi_1)-6)!$ and so $(6n-6)!$ as well. We conclude that $\Per \phi =\Per \phi_1 =\Fix \phi_1^s =\Fix \phi^s \leqslant \Fix \phi^{(6n-6)!}\leqslant \Per \phi$ and so, $\Per \phi =\Fix \phi^{(6n-6)!}$.
\end{proof}

\begin{rem}
Modulo missing details, this fact was implicitly contained in an older result by M. Takahasi, who proved that an ascending chain of subgroups of a free group, with rank uniformly bounded above by a fixed constant (like the $\Fix \psi^p$'s), must stabilize; see~\cite[p.~114]{LS}.
\end{rem}

We close the present section by extending this same result to the context of free-abelian times free groups.

\begin{thm}\label{periodic-fatf}
There exists a computable constant $C_3=C_3(m,n)$ such that $\Per \Psi=\Fix \Psi^{C_3}$, for every $\Psi\in \End(\Z^m\times F_n)$.
\end{thm}

\begin{proof}
Delgado--Ventura~\cite[Prop.~5.1]{DV} gave a classification of all endomorphisms of $G=\Z^m\times F_n$ in two types. For those of the second type, say $\Psi_{z,l,h,Q,P}$ (see~\cite{DV} for the notation), it is clear that the subgroup $\langle z, \Z^m\rangle\leqslant \Z^m\times F_n$ is invariant under $\Psi$ (denote $\Psi_| \colon \langle z, \Z^m\rangle \to \langle z, \Z^m\rangle$ its restriction), and it contains $\im \Psi$. Therefore, by Proposition~\ref{periodic-abelian}, $\Per\Psi =\Per\Psi_| =\Fix (\Psi_|)^{L_3(m+1)} =\Fix \Psi^{L_3(m+1)}$, since $\langle z, \Z^m\rangle \simeq \Z^{m+1}$ is abelian. Thus, the computable constant $C_3(n,m)=L_3(m+1)$ satisfies the desired result for all endomorphisms of the second type.

Suppose now that $\Psi$ is of the first type, i.e., $\Psi=\Psi_{\phi, Q, P}$, where $\phi\in \End(F_n)$, $Q\in M_{m\times m}(\Z)$, and $P\in M_{n\times m}(\Z)$. By Propositions~\ref{periodic-abelian} and~\ref{periodic-free}, we know that $\Per Q=\Fix Q^{L_3}$ and $\Per \phi =\Fix \phi^{(6n-6)!}$ for some computable constant $L_3=L_3(m)$. Take $C_3(m,n)=\lcm \big( L_3(m), (6n-6)!\big)$ and let us prove that $\Per \Psi =\Fix \Psi^{C_3}$.

By construction, we have both $\Per Q=\Fix Q^{C_3}$ and $\Per \phi =\Fix \phi^{C_3}$. It remains to see that the matrix $P$ does not affect negatively into the calculations. To prove $\Per \Psi =\Fix \Psi^{C_3}$, it is enough to see that $\Fix \Psi^k \leqslant \Fix \Psi^{C_3}$ for all $k\geqslant 1$, which reduces to see that $\Fix \Psi^{\lambda C_3} \leqslant \Fix \Psi^{C_3}$ for every $\lambda\in \N$ (in fact, if this is true then $\Fix \Psi^k \leqslant \Fix \Psi^{kC_3}\leqslant \Fix \Psi^{C_3}$, for an arbitrary $k\geqslant 1$).

By Delgado--Ventura~\cite[Lemma~5.4(ii)]{DV}, powers work like this: $(\Psi_{\phi, Q, P})^k=\Psi_{\phi^k, Q^k, P_k}$, where $P_k=\sum_{i=0}^{k-1} A^i PQ^{(k-1)-i}$ and $A\in M_{n\times n}(\Z)$ is the abelianization matrix corresponding to $\phi\in \End(F_n)$. In our situation, $(\Psi_{\phi, Q, P})^{C_3}=\Psi_{\phi^{C_3}, Q^{C_3}, P_{C_3}}$, and $(\Psi_{\phi, Q, P})^{\lambda C_3}= \Psi_{\phi^{\lambda C_3}, Q^{\lambda C_3}, P_{\lambda C_3}}$, where

 \begin{equation}\label{prs}
\begin{array}{ccl}
P_{\lambda C_3} & = & \sum_{i=0}^{\lambda C_3-1} A^i PQ^{(\lambda C_3-1)-i} \\[1mm] & = & \sum_{j=0}^{\lambda-1} \sum_{i=0}^{C_3-1} A^{jC_3 +i} PQ^{(\lambda C_3 -1)-(jC_3 +i)} \\[1mm] &  = & \sum_{j=0}^{\lambda-1} \sum_{i=0}^{C_3-1} A^{jC_3 +i} PQ^{(\lambda-j)C_3 -1-i} \\[1mm] & = & \sum_{j=0}^{\lambda-1} A^{jC_3} \big( \sum_{i=0}^{C_3-1} A^i PQ^{(C_3-1)-i}\, \big) Q^{(\lambda-j-1)C_3} \\[1mm] & = & \sum_{j=0}^{\lambda-1} (A^{C_3})^j P_{C_3} (Q^{C_3})^{(\lambda-1)-j}.
\end{array}
 \end{equation}

Take any element $t^a u \in \Fix \Psi^{\lambda C_3}$ and let us prove that $t^a u \in \Fix \Psi^{C_3}$. Our assumption means that $t^{aQ^{\lambda C_3}+u^{\rm ab}P_{\lambda C_3}}(u\phi^{\lambda C_3})=t^a u$ and so,
\begin{itemize}
\item[(1)] $a(I_m-Q^{\lambda C_3})=u^{\rm ab}P_{\lambda C_3}$, and
\item[(2)] $u\in \Fix \phi^{\lambda C_3} \leqslant \Per \phi=\Fix \phi^{C_3}$; in particular, $u^{\rm ab}A^{C_3}=u^{\rm ab}$.
\end{itemize}
Now from~\eqref{prs} and condition $(1)$ we have,
 $$
\begin{array}{ccl}
a(I_m-Q^{C_3})(I+Q^{C_3}+ \cdots + Q^{(\lambda -1)C_3})& = & u^{\rm ab}\sum_{j=0}^{\lambda-1} (A^{C_3})^j P_{C_3} (Q^{C_3})^{(\lambda-1)-j} \\[1mm] & = & u^{\rm ab}\sum_{j=0}^{\lambda-1} P_{C_3} (Q^{C_3})^{(\lambda-1)-j} \\[1mm] & = & u^{\rm ab}P_{C_3}\sum_{j=0}^{\lambda-1} (Q^{C_3})^{(\lambda-1)-j} \\[1mm] & = & u^{\rm ab} P_{C_3} \big( I+Q^{C_3}+ \cdots +Q^{(\lambda -1)C_3} \big),
\end{array}
 $$
which means that $a(I_m-Q^{C_3})-u^{\rm ab}P_{C_3} \in \ker \big( I_m +Q^{C_3}+ \cdots +Q^{(\lambda -1)C_3} \big)$. But
 $$
\ker \big( I_m+Q^{C_3}+ \cdots +Q^{(\lambda -1)C_3} \big)\leqslant \ker (I_m-Q^{\lambda C_3})=\Fix Q^{\lambda C_3}\leqslant \Per Q =\Fix Q^{C_3}=\ker (I_m-Q^{C_3})
 $$
hence, we also have $a(I_m-Q^{C_3})-u^{\rm ab}P_{C_3} \in \ker (I_m -Q^{C_3})$. However, the two polynomials $1+x^{C_3}+ \cdots + x^{(\lambda -1)C_3}$ and $1-x^{C_3}$ are relatively prime so, from Bezout's equality we deduce that $\ker \big( I_m+Q^{C_3}+\cdots +Q^{(\lambda -1)C_3}) \cap \ker (I_m -Q^{C_3})=\{0\}$. Therefore, $a(I_m-Q^{C_3})-u^{\rm ab}P_{C_3}=0$ and so,
 $$
(t^a u)\Psi^{C_3}=t^{aQ^{C_3}+\textbf{u}P_{C_3}}(u\phi^{C_3})=t^a u.
 $$
This shows that $\Fix \Psi^{\lambda C_3} =\Fix \Psi^{C_3}$ for every $\lambda \in \N$, from which we immediately deduce $\Per \Psi=\Fix \Psi^{C_3}$. This means that the constant $C_3(n,m)=\lcm \big( L_3(m), (6n-6)!\big)$ satisfies the desired result for all endomorphisms of the first type.

Hence, the computable constant $C_3(n,m)=\lcm \big( L_3(m), L_3(m+1), (6n-6)!\big)$ makes the job.
\end{proof}

\begin{cor}
Let $\Psi \in \End(\Z^m \times F_n)$. Then $\Per \Psi$ is finitely generated if and only if $\Fix \Psi^p$ is finitely generated for all $p\geqslant 1$.
\end{cor}

\begin{proof}
This follows immediately from Theorem~\ref{periodic-fatf} and Corollary~\ref{aa}.
\end{proof}

\section{The auto-fixed closure of a subgroup of $\Z^m\times F_n$}\label{5}

Given an endomorphism, it is natural to ask for the computability of (a basis of) its fixed subgroup (or its periodic subgroup). In the abelian case, this can easily be done by just solving a system of linear equations, because the fixed point subgroup of an endomorphism of $\Z^m$ is nothing else but the eigenspace of eigenvalue 1 of the corresponding matrix, $\Fix Q=E_1(Q)$.

In the free case, this is a hard problem solved for automorphisms by making strong use of the train track techniques, see Bogopolski--Maslakova~\cite{BM} (amending the previous wrong version Maslakova~\cite{Ma}) and, alternatively, Feingh--Handel~\cite[Prop.~7.7]{FH}.

\begin{thm}[Bogopolski--Maslakova, \cite{BM}; Feingh--Handel, \cite{FH}]\label{maslakova}
Let $\phi \colon F_n\to F_n$ be an automorphism. Then, a free-basis for $\Fix \phi$ is computable.
\end{thm}

Finally, the free-abelian times free case was studied by Delgado--Ventura who solved the problem (including the decision on whether the fixed subgroup is finitely generated or not), modulo a solution for the free case. More precisely,

\begin{thm}[Delgado--Ventura, \cite{DV}]\label{compute-fix}
Let $G=\Z^m \times F_n$. There is an algorithm which, on input an automorphism $\Psi \colon G\to G$, decides whether $\Fix \Psi$ is finitely generated or not and, if so, computes a basis for it.
\end{thm}

We note that Theorems~\ref{maslakova} and~\ref{compute-fix} work for automorphisms; as far as we know, the computability of the fixed subgroup of an endomorphism, both in the free and in the free-abelian times free cases, remains open.

%

In the present section, we are interested in the dual problems: given a subgroup, decide whether it can be realized as the fixed subgroup of an endomorphism (resp., an automorphism, a family of endomorphisms, a family of automorphisms) and in the affirmative case, compute such an endomorphism (resp., automorphism, family of endomorphisms, family of automorphisms).

Generalizing the terminology introduced in Martino--Ventura~\cite{MV} to an arbitrary group $G$, a subgroup $H\leqslant G$ is called \emph{endo-fixed} (resp., \emph{auto-fixed}) if $H=\Fix S$ for some set of endomorphisms $S\subseteq \End(G)$ (resp., automorphisms $S\subseteq \Aut(G)$). Simillarly, a subgroup $H\leqslant G$ is said to be \emph{1-endo-fixed} (resp., \emph{1-auto-fixed}) if $H=\Fix \phi$, for some $\phi\in \End(G)$ (resp., some $\phi\in \Aut(G)$). Notice that an auto-fixed (resp., endo-fixed) subgroup of $G$ is an intersection of 1-auto-fixed (resp., 1-endo-fixed) subgroups of $G$, and vice-versa.

Of course, it is straightforward to see that all these notions do coincide in the abelian case: a subgroup $H\leqslant \Z^m$ is endo-fixed if and only if it is auto-fixed, if and only if it is 1-endo-fixed, if and only if it is 1-auto-fixed, and if and only if it is a direct summand, $H\leqslant_{\oplus}\Z^m$.

In the free case (and so, in the free-abelian times free as well) the situation is much more delicate: in Martino--Ventura~\cite{MV}, the authors conjectured that the families of auto-fixed and 1-auto-fixed subgroups of $F_n$ do coincide; in other words, the family of 1-auto-fixed subgroups of $F_n$ is closed under arbitrary intersections. (A similar conjecture can be stated for endomorphisms.) As far as we know, this still remains an open problem, with no progress made since the paper~\cite{MV} itself, where the authors showed that, for any submonoid $S\leqslant \End (F_n)$, there exists $\phi \in S$ such that $\Fix(S)$ is a free factor of $\Fix \phi$; however, they also gave an explicit example of a 1-auto-fixed subgroup of $F_n$ admitting a free factor which \emph{is not} even endo-fixed. In this context it is worth mentioning the result Martino--Ventura~\cite[Cor.~4.2]{MV2} showing that we can always restrict ourselves to consider finite intersections.

Let $H\leqslant G$. We denote by $\Aut_H(G)$ the subgroup of $\Aut(G)$ consisting of all automorphisms of $G$ which fix $H$ pointwise,
$\Aut_H(G)=\{\phi \in \Aut(G) \mid H\leqslant \Fix \phi\}$, usually called the \emph{(pointwise) stabilizer} of $H$. Analogously, we denote by $\End_H(G)$ the submonoid of $\End(G)$ consisting of all endomorphisms of $G$ which fix every element of $H$. Clearly, $\Aut_H(G) \leqslant \End_H(G)$. The following is a well-known result about stabilizers in the free group case, which will be used later:

\begin{thm}[{McCool, \cite{McCool}; see also~\cite[Prop.~I.5.7]{LS}}]\label{McCool}
Let $H\leqslant_{fg} F_n$, given by a finite set of generators. Then the stabilizer, $\Aut_H (F_n)$, of $H$ is also finitely generated (in fact, finitely presented), and a finite set of generators (and relations) is algorithmically computable.
\end{thm}

Following with the terminology from~\cite{MV}, the \emph{auto-fixed closure} of $H$ in $G$, denoted $\acl_G{(H)}$, is the subgroup
 $$
\acl_G{(H)}=\Fix(\Aut_H(G))=\bigcap_{{\tiny \begin{array}{c} \phi\in \Aut(G) \\ H\leqslant \Fix \phi \end{array}}} \Fix\phi ,
 $$
i.e., the smallest auto-fixed subgroup of $G$ containing $H$. Similarly, the \emph{endo-fixed closure} of $H$ in $G$, is $\ecl_G{(H)}=\Fix(\End_H(G))$. Since $\Aut_H(G)\leqslant \End_H(G)$, it is obvious that $\ecl_{G}(H) \leqslant \acl_{G}(H)$. However, the equality does not hold in general (for example, the free group $F_n$ admit 1-endo-fixed subgroups which are not auto-fixed; see Martino--Ventura~\cite{MV3}).

In Ventura~\cite{V}, fixed closures in free groups are studied from the algorithmic point of view. More precisely, the following results were proven:

\begin{thm}[Ventura, \cite{V}]\label{Enric}
Let $H\leqslant_{fg} F_n$, given by a finite set of generators. Then, a free-basis for the auto-fixed closure $\acl_{F_n}(H)$ (resp., the endo-fixed closure $\ecl_{F_n}(H)$) of $H$ is algorithmically computable, together with a set of $k\leqslant 2n$ automorphisms $\phi_1, \ldots, \phi_k \in \Aut(F_n)$ (resp., endomorphisms $\phi_1, \ldots, \phi_k \in \End(F_n)$), such that $\acl_{F_n}(H)= \Fix \phi_1 \cap \cdots \cap \Fix \phi_k$ (resp., $\ecl_{F_n}(H)= \Fix \phi_1 \cap \cdots \cap \Fix \phi_k$).
\end{thm}

\begin{cor}[Ventura, \cite{V}]\label{Enric-deciding}
It is algorithmically decidable whether a given $H\leqslant_{fg} F_n$ is auto-fixed (resp., endo-fixed) or not.
\end{cor}

For example it is well known that, for every $w\in F_n$ and $r\in \Z$, the equation $x^r=w^r$ has a \emph{unique} solution in $F_n$, which is the obvious one $x=w$; this means that any endomorphism $\phi\colon F_n \to F_n$ fixing $w^r$ \emph{must} also fix $w$. Therefore, the auto-fixed and endo-fixed closures of a cyclic subgroup of $F_n$ are equal to the maximal cyclic subgroup where it is contained; in other words, a cyclic subgroup of $F_n$ is auto-fixed, if and only if it is endo-fixed, and if and only if it is maximal cyclic.

In the present section, we prove the analog of Theorem~\ref{Enric} for free-abelian time free groups, and only in the automorphism case. Our main results in the section are:

\begin{thm}\label{closure}
Let $G=\Z^m\times F_n$. There is an algorithm which, given a finite set of generators for a subgroup $H\leqslant_{fg} G$, outputs a set of automorphisms $\Psi_1, \ldots, \Psi_k \in \Aut(G)$ such that $\acl_{G}(H)= \Fix \Psi_1 \cap \cdots \cap \Fix \Psi_k$, decides whether this is finitely generated or not and, in case it is, computes a basis for it. \end{thm}

\begin{cor}\label{deciding}
One can algorithmically decide whether a given $H\leqslant_{fg} G$ is auto-fixed or not, and in case it is, compute a set of automorphisms $\Psi_1, \ldots, \Psi_k \in \Aut(G)$ such that $H=\Fix \Psi_1 \cap \cdots \cap \Fix \Psi_k$.
\end{cor}

We want to emphasize that we did not succeed in the task of constructing an example of a finitely generated subgroup $H\leqslant_{fg} G=\Z^m\times F_n$ such that $\acl_G(H)$ is \emph{not} finitely generated; it could be that such examples do not exist so the following is an interesting open question:

\begin{que}
Is it true that, for every $H\leqslant_{fg} G=\Z^m\times F_n$, the auto-fixed closure $\acl_G(H)$ is again finitely generated ? What about the endo-fixed closure $\ecl_G(H)$ ?
\end{que}

To prove Theorem~\ref{closure} and Corollary~\ref{deciding}, we plan to follow the same strategy as in the free case, which is conceptually very easy: given $H\leqslant_{fg} F_n$, use Theorem~\ref{McCool} to compute a set of generators for the stabilizer, say $\Aut_H(F_n)=\langle \phi_1, \ldots ,\phi_k\rangle$, then use Theorem~\ref{maslakova} to compute $\Fix \phi_i$ for each $i=1,\ldots ,k$, and finally intersect them all in order to get the auto-fixed closure, $\acl_{F_n}(H)=\Fix \phi_1 \cap \cdots \cap \Fix \phi_k$ (the bound $k\leqslant 2n$ comes from free group arguments and will be lost in the more general free-abelian times free context).

To make this strategy work in the free-abelian times free case, we have to overcome two extra difficulties not present at the free case:
\begin{itemize}
\item[(1)] We need an analog to McCool's result for the group $\Z^m \times F_n$; stabilizers are going to be still finitely presented and computable, but more complicated than in the free case. The natural approach to this problem, trying to analyze directly how does an automorphism in $\Aut_H(G)$ look like, brings to a tricky matrix equation with which we were unable to solve the problem; instead, our approach will be indirect, making use of another two more powerful results from the literature.
\item[(2)] When trying to compute $\Fix \Psi_1 \cap \cdots \cap \Fix \Psi_k$, it may very well happen that some of the individual $\Fix \Psi_i$'s are not finitely generated; in this case, Theorem~\ref{compute-fix} recognizes this fact and stops, giving us nothing else, while we still have to decide whether the full intersection $\Fix \Psi_1 \cap \cdots \cap \Fix \Psi_k$ is finitely generated or not (and compute a basis for it in case it is so).
\end{itemize}

We succeed overcoming these two difficulties in Theorem~\ref{main-1} and Proposition~\ref{com-int}, respectively.

The versions of Theorem~\ref{closure} and Corollary~\ref{deciding} for endomorphisms seem to be much more tricky and remain open (their versions for the free group, contained in Theorem~\ref{Enric} and Corollary~\ref{Enric-deciding}, are already much more complicated because the monoid $\End_{F_n}(H)$ is not necessarily finitely generated, even with $H$ being so, and also computability of fixed subgroups is not known for endomorphisms).

\begin{que}
Let $G=\Z^m\times F_n$. Is there an algorithm which, given a finite set of generators for a subgroup $H\leqslant_{fg} G$, decides whether
\begin{itemize}
\item[(i)] the monoid $\End_H(G)$ is finitely generated or not and, in case it is, computes a set of endomorphisms $\Psi_1, \ldots, \Psi_k \in \End(G)$ such that $\End_{H}(G)=\langle \Psi_1, \ldots, \Psi_k \rangle$ ?
\item[(ii)] $\ecl_G(H)$ is finitely generated or not and, in case it is, computes a basis for it ?
\item[(iii)] $H$ is endo-fixed or not ?
\end{itemize}
\end{que}

Let us begin by understanding stabilizers in $G=\Z^m\times F_n$. For this, we need to remind a couple of other results from the literature.

Given a tuple of conjugacy classes $W=([g_1], \ldots ,[g_k])$ from a group $G$, the stabilizer of $W$, denoted $\Aut_W(G)$, is the group of automorphisms fixing all the $[g_i]$'s, i.e., sending the elements $g_i$ to conjugates of themselves (with possibly different conjugators); more precisely,
 $$
\Aut_W(G) =\{\phi\in \Aut(G) \mid g_1\phi \sim g_1,\ldots ,g_k\phi\sim g_k\},
 $$
where $\sim$ stands for conjugation in $G$ ($g\sim h$ if and only if $g=x^{-1}hx=h^x$ for some $x\in G$). Of course, if $H=\langle h_1,\ldots ,h_k\rangle \leqslant_{fg} G$, and $W=([h_1],\ldots ,[h_k])$, then $\Aut_H(G)\leqslant \Aut_W(G)$, without equality, in general.

McCool's Theorem~\ref{McCool} was a variation and an extension of a much earlier result: back in the 1930's, Whitehead already solved the  orbit problem for conjugacy classes in the free group: given two tuples of conjugacy classes $V=([v_1], \ldots ,[v_k])$ and $W=([w_1],\ldots ,[w_k])$ in $F_n$, one can algorithmically decide whether there is an automorphism $\phi\in \Aut(F_n)$ such that $v_i\phi \sim w_i$, for every $i=1,\ldots ,k$; see~\cite[Prop.~4.21]{LS} or~\cite{W}; this was based in the so-called Whitehead automorphisms and the peak reduction technique. McCool's work 40 years later consisted on (1) deducing as a corollary that $\Aut_W(F_n)$ if finitely presented and a finite presentation is computable from the given $W$; and (2) extending everything to real elements instead of conjugacy classes and so, getting a solution to the orbit problem for tuples of elements, and the finite presentability (and computability) for stabilizers of subgroups, stated in Theorem~\ref{McCool}.

Much more recently, a new version of these peak reduction techniques has been developed by M. Day~\cite{Day} for right-angled Artin groups, extending McCool result (1) above to this bigger class of groups; we are interested in the stabilizer part:

\begin{thm}[{Day, \cite[Thm.~1.2]{Day}}]\label{Day}
There is an algorithm that takes in a tuple $W$ of conjugacy classes from a right-angled Artin group $A(\Gamma)$ and produces a finite presentation for its stabilizer $\Aut_W (A(\Gamma))$.
\end{thm}

Of course, we can make good use of Day's result in our case, because free-abelian times free groups are (a very special kind of) right-angled Artin groups; namely, $\Z^m\times F_n =A(\Gamma_{m,n})$ where $\Gamma_{m,n}$ is the complete graph on $m$ vertices and the null graph on $n$ vertices, together with $mn$ edges joining each pair of vertices one in each side. The problem in doing this is that Day's result works only for conjugacy classes and the corresponding result for real elements is not known in general for right-angled Artin groups; while we need the finite generation (and computability) of stabilizers of \emph{subgroups} in $\Z^m\times F_n$. We overcome this difficulty by using a result from Bogopolski--Ventura~\cite{BV} relating stabilizers of subgroups and of tuples of conjugacy classes, in torsion-free hyperbolic groups:

\begin{thm}[{Bogopolski--Ventura~\cite[Thm.~1.2]{BV}}]\label{BV}
Let $G$ be a torsion-free $\delta$-hyperbolic group with respect to a finite generating set $S$. Let $g_1, \ldots ,g_r$ and $g'_1,\ldots ,g'_r$ be elements of $G$ such that $g_i\sim g'_i$ for every $i=1,\ldots ,r$. Then, there is a uniform conjugator for them if and only if $w(g_1,\ldots ,g_r)\sim w(g'_1,\ldots ,g'_r)$ for every word $w$ in $r$ variables and length up to a computable constant $C=C(\delta, |S|, \sum_{i=1}^r |g_i|)$, depending only on $\delta$, $|S|$, and $\sum_{i=1}^r |g_i|$.
\end{thm}

Using these results we can effectively compute generators for the stabilizer of a given subgroup $H\leqslant_{fg}\Z^m\times F_n$. For our purposes, we do not need at all any set of relations; however, for completeness with respect to Day's result, we further prove that these stabilizers are also finitely presented and compute a full set of relations (postponing this part of the proof to Appendix~\ref{6}).

\begin{thm}\label{main-1}
Let $H\leqslant_{fg} G=\Z^m\times F_n$, given by a finite set of generators. Then the stabilizer, $\Aut_H (G)$, of $H$ is finitely presented, and a finite set of generators and relations is algorithmically computable.
\end{thm}

\begin{proof}

From the given set of generators, compute a basis for $H$, say $\{t^{a_1}u_1, \ldots, t^{a_r}u_r, t^{b_1}, \ldots, t^{b_s}\}$; in particular, we have a free-basis $\{u_1,\ldots, u_r\}$ for $H\pi$, and an abelian basis $\{t^{b_1}, \ldots, t^{b_s}\}$ for $L_H=H\cap \Z^m$.

If $r=0$ then $H=L_H$ and, clearly, $\Psi_{\phi,Q,P}\in \Aut_H(G)$ if and only if $Q\in \Aut_{L_H}(\Z^m)$. So, $\Aut_H(G)$ is generated by the following finite set of automorphisms of $G$: (1) $\Psi_{\phi, I_m, 0}$, with $\phi$ running over the Nielsen automorphisms of $F_n$; (2) $\Psi_{id, Q, 0}$, with $Q$ running over the generators of $\Aut_{L_H}(\Z^m)$ computed by Theorem~\ref{Day} (note that, since $\Z^m$ is abelian, $\Aut_{L_H}(\Z^m)=\Aut_{([b_1], \ldots ,[b_s])}(\Z^m)$); and (3) $\Psi_{id, I_m, 1_{i,j}}$, with $1_{i,j}$ being the zero $n\times m$ matrix with a single 1 at position $(i,j)$, $i=1,\ldots ,n$, $j=1,\ldots ,m$. The computation of finitely many relations on these generators determining a presentation for $\Aut_H(G)$ is postponed to the Appendix~\ref{rels}.

Assume that $r=\rk (H\pi)\geqslant 1$. Apply Theorem~\ref{BV} to the free group $F_n$ and words $u_1,\ldots ,u_r$, and compute the constant $C=C(0,n, \sum_{i=1}^r |u_i|)$. Consider the tuple of elements from $G$ given by $W=\big( w_1(t^{a_1}u_1, \ldots, t^{a_r}u_r), \ldots ,w_{M}(t^{a_1}u_1, \ldots, t^{a_r}u_r), t^{b_1},\ldots ,t^{b_s}\big)$, where $w_1, \ldots ,w_{M}$ is the sequence (in any order) of all reduced words on $r$ variables and of length up to $C$. We claim that
 \begin{equation}\label{=}
\Aut_W(G)=\Aut_{H}(G)\cdot \Inn(G).
 \end{equation}
In fact, the inclusion $\geqslant$ is obvious. To see $\leqslant\,$, take $\Psi =\Psi_{\phi,Q,P}\in \Aut_W(G)$, that is, an automorphism $\Psi$ satisfying $w_i(t^{a_1}u_1, \ldots, t^{a_r}u_r)\Psi \sim w_i(t^{a_1}u_1, \ldots, t^{a_r}u_r)$ for $i=1,\ldots ,M$, and $t^{b_j}\Psi \sim t^{b_j}$ for $j=1,\ldots ,s$. We have $t^{b_j}\Psi =t^{b_j}$ (since these are central elements from $G$), and $w_i(u_1, \ldots, u_r)\phi \sim w_i(u_1, \ldots, u_r)$ so, by Theorem~\ref{BV}, $w_i(u_1, \ldots, u_r)\phi =x^{-1} w_i(u_1, \ldots, u_r)x$ for a common conjugator $x\in F_n$; in particular, $u_i\phi =x^{-1} u_i x$ for $i=1,\ldots ,r$ and so, $\phi=(\phi \gamma_{x^{-1}})\gamma_x$, with $\phi \gamma_{x^{-1}}\in \Aut_{H\pi}(F_n)$. Therefore, $\Psi=(\Psi \Gamma_{x^{-1}})\Gamma_x$, with $\Psi \Gamma_{x^{-1}}\in \Aut_{H}(G)$.

Now, by Theorem~\ref{Day}, this stabilizer is finitely presented and a finite presentation
 \begin{equation}\label{pres}
\Aut_W(G)=\langle \Psi_1, \ldots ,\Psi_{\ell} \mid R_1, \ldots ,R_d \rangle
 \end{equation}
can be computed, where the $\Psi_i$'s are explicit automorphisms of $G$, and the $R_j$'s are words on them satisfying $R_j(\Psi_1, \ldots ,\Psi_{\ell})=\id_G$. From the previous paragraph, we can algorithmically rewrite $\Psi_i =\Psi'_i \Gamma_{x_i}$ for some $\Psi'_i\in \Aut_H(G)$ and some $x_i\in F_n$, $i=1,\ldots ,\ell$ (note that some $\Psi'_i$ could be the identity, corresponding to $\Psi_i$ being possibly a genuine conjugation of $G$). Finally, let us distinguish two cases.

Suppose $r=\rk (H\pi)\geqslant 2$. We claim that $\Aut_H(G)=\langle \Psi'_1, \ldots ,\Psi'_{\ell} \rangle$: the inclusion $\geqslant$ is trivial; for the other, take $\Psi\in \Aut_H(G)\leqslant \Aut_W(G)$ and, since $\Inn(G)$ is a normal subgroup of $\Aut(G)$, we have $\Psi= w(\Psi_1, \ldots ,\Psi_{\ell})= w(\Psi'_1\Gamma_{x_1},\ldots ,\Psi'_{\ell} \Gamma_{x_{\ell}})= w(\Psi'_1,\ldots ,\Psi'_{\ell}) \Gamma_x$ for some $x\in F_n$. But both $\Psi$ and $w(\Psi'_1,\ldots ,\Psi'_{\ell})$ fix $t^{a_1}u_1,\ldots ,t^{a_r}u_r$ and $r\geqslant 2$ so, $x=1$ and $\Psi=w(\Psi'_1,\ldots ,\Psi'_{\ell})\in \langle \Psi'_1, \ldots ,\Psi'_{\ell}\rangle$.

Suppose now that $r=\rk (H\pi)=1$. The argument in the previous paragraph tells us that $\Aut_H(G)=\langle \Psi'_1, \ldots ,\Psi'_{\ell},\, \Gamma_{\hat{u}_1} \rangle$, where $\hat{u}_1$ is the root of $u_1$ in $F_n$, i.e., the unique non-proper power in $F_n$ such that $u_1=\hat{u}_1^{\alpha}$ for $\alpha>0$ (since now, in the last part of the argument, $x$ only commutes with $u_1\neq 1$).

Up to here we have proved that $\Aut_H(G)$ is finitely generated and a finite set of generators is algorithmically computable. We postpone the argument about relations to the Appendix~\ref{rels}.
\end{proof}

Now we turn to the computability of fixed points by a given collection of automorphisms.

\begin{prop}\label{com-int}
Let $G=\Z^m\times F_n$. There is an algorithm which, given $\Psi_1,\ldots ,\Psi_k\in \Aut(G)$, it decides whether $\Fix \Psi_1  \cap \cdots \cap \Fix \Psi_k$ is finitely generated or not and, in the affirmative case, computes a basis for it.
\end{prop}

\begin{rem}
Two related results are Theorem~\ref{compute-fix} above, and Theorem~\cite[Thm.~4.8]{DV}. With the first one we can decide whether each $\Fix \Psi_i$ is finitely generated and, in this case, compute a basis; and with the second, assuming $\Fix \Psi_i$ and $\Fix \Psi_j$ finitely generated, we can decide whether $\Fix \Psi_i \cap \Fix \Psi_j$ is finitely generated again and, in this case, compute a basis for it. However, these two results combined in an induction argument are not enough to prove Proposition~\ref{com-int} because it could very well happen that some of the individual $\Fix \Psi_i$'s (even a partial intersection of some of them) is not finitely generated while $\Fix \Psi_1 \cap \cdots \cap \Fix \Psi_k$ is so. Thus, we are going to adapt the proof of Theorem~\ref{compute-fix} to compute directly the fixed subgroup of a finite tuple of automorphisms, without making reference to the fixed subgroup of each individual one. \end{rem}

\begin{proof}[Proof of Proposition~\ref{com-int}]
Write $\Psi_i =\Psi_{\phi_i, Q_i, P_i}\colon G\to G$, $t^a u \mapsto t^{aQ_i+u\rho P_i}u\phi_i$, for some $\phi_i\in \Aut(F_n)$, $Q_i\in \GL_m(\Z)$, and $P_i\in \Mat_{n\times m}(\Z)$, $i=1,2, \ldots, k$, where $\rho\colon F_n \twoheadrightarrow \Z^n$ is the abelianization map. We have
 $$
\begin{array}{rcl}
\Fix \Psi_1 \cap \cdots \cap \Fix \Psi_k & \!\!\!=\!\!\! & \left\{ t^a u\in G \mid u\in \cap_{i=1}^k \Fix \phi_i ,\,\, a(I_m -Q_i)=u\rho P_i,\,\, i=1,\ldots, k\right\} \\[1mm] & = & \{ t^a u\in G \mid u\in \cap_{i=1}^k \Fix \phi_i ,\,\, a(I_m -Q_1 | \cdots | I_m -Q_k)=u\rho (P_1 | \cdots | P_k)\},
\end{array}
 $$
were $(I_m-Q_1 | \cdots | I_m-Q_k)\in \Mat_{m\times km}(\Z)$ and $(P_1 | \cdots | P_k)\in \Mat_{n\times km}(\Z)$ are the indicated concatenated matrices, corresponding to linear maps $\tilde{Q}\colon \Z^m \to \Z^{km}$ and $\tilde{P}\colon \Z^n \to \Z^{km}$, respectively.


Let $\rho'$ be the restriction of $\rho$ to $\Fix \phi_1 \cap \cdots \cap \Fix \phi_k$ (not to be confused with the abelianization map of the subgroup $\Fix \phi_1 \cap \cdots \cap \Fix \phi_k$ itself), let $\tilde{P}'$ be the restriction of $\tilde{P}$ to $\im \rho'$; let $M=\im \tilde{Q}\leqslant \Z^{km}$, let $N=M\cap \im \tilde{P}'$, and consider the preimages of $N$ first by $\tilde{P}'$ and then by $\rho'$, see the following diagram:
 \begin{equation*}\label{Pic 2}
\xy
(60,0)*+{\geqslant M = \operatorname{Im} \tilde{Q}};
(59,-17.3)*+{=M\cap \im {\tilde{P}'}.};
(0,-4.5)*+{\rotatebox[origin=c]{90}{$\leqslant$}};
(24,-4.5)*+{\rotatebox[origin=c]{90}{$\unlhd$}};
(46,-4.5)*+{\rotatebox[origin=c]{90}{$\unlhd$}};
{\ar@{->>}^-{\rho} (0,0)*++{F_{n}}; (24,0)*++{\ZZ^{n}}};
{\ar^-{\tilde{P}} (24,0)*++++{}; (46,0)*++{\ZZ^{km}}};
{\ar@{->>}^-{\rho'} (-5,-9)*++{\Fix \phi_1 \cap \cdots \cap \Fix \phi_k}; (24,-9)*++{\im \rho'}};
{\ar@{->>}^-{{\tilde{P}'}} (24,-9)*+++++{}; (46,-9)*++{\im {\tilde{P}'}}};
(0,-13.5)*+{\rotatebox[origin=c]{90}{$\unlhd$}};
(24,-13.5)*+{\rotatebox[origin=c]{90}{$\unlhd$}};
(46,-13.5)*+{\rotatebox[origin=c]{90}{$\unlhd$}};
{\ar@{|->} (46,-18)*+++{N }; (24,-18)*++{N \tilde{P}^{\prime -1}}};
{\ar@{|->} (24,-18)*+++++++{}; (0,-18)*+{N \tilde{P}^{\prime -1} \rho'^{-1} }};
\endxy
 \end{equation*}
We claim that $(\Fix \Psi_1 \cap \cdots \cap \Fix \Psi_k )\pi =N\tilde{P}^{\prime -1} \rho'^{-1}$. In fact, for $u\in (\Fix \Psi_1 \cap \cdots \cap \Fix \Psi_k)\pi$, there exists $a\in \Z^m$ such that $t^a u\in \Fix \Psi_1 \cap \cdots \cap \Fix \Psi_k$, i.e., $u\in \Fix \phi_i$ and $a(I_m-Q_i)=u\rho P_i$, $i=1,\ldots ,k$. So, $u\in \Fix \phi_1\cap \cdots \cap \Fix \phi_k$ and $u\rho'\tilde{P}' =a\tilde{Q}\in M\cap \im \tilde{P}'$ and hence, $u\in N\tilde{P}^{\prime -1} \rho'^{-1}$. On the other hand, for $u\in N\tilde{P'}^{-1} \rho'^{-1}$, we have $u\in \Fix \phi_1 \cap \cdots \cap \Fix \phi_k$ and $u\rho'\tilde{P}' \in N\leqslant M=\im \tilde{Q}$ so, $u\rho \tilde{P}' =a\tilde{Q}$ for some $a\in \Z^m$; this means that $t^a u\in \Fix \Psi_1 \cap \cdots \cap \Fix \Psi_k$ and hence $u\in (\Fix \Psi_1 \cap \cdots \cap \Fix \Psi_k)\pi$. This proves the claim.

Now $\Fix \Psi_1 \cap \cdots \cap \Fix \Psi_k\leqslant G$ is finitely generated if and only if $(\Fix \Psi_1 \cap \cdots \cap \Fix \Psi_k)\pi =N\tilde{P}^{\prime -1} \rho'^{-1}$ is finitely generated, which (since it is a normal subgroup) happens if and only if $N\tilde{P}^{\prime -1} \rho'^{-1}$ is trivial (i.e., $\Fix \phi_1 \cap \cdots \cap \Fix \phi_k=\langle u\rangle$ with $u\rho\neq 0$ and $N=\{ 0\}$) or of finite index in $\Fix \phi_1 \cap \cdots \cap \Fix \phi_k$. That is, $\Fix \Psi_1 \cap \cdots \cap \Fix \Psi_k$ is finitely generated if and only if
\begin{itemize}
\item[(i)] $\Fix \phi_1 \cap \cdots \cap \Fix \phi_k=\langle u\rangle$ with $u\rho \neq 0$ and $N=\{0\}$, or
\item[(ii)] $[\im P' : N]=[\im \rho' : N\tilde{P}^{\prime -1}]=[\Fix \phi_1 \cap \cdots \cap \Fix \phi_k : N\tilde{P}^{\prime -1} \rho'^{-1}]<\infty$ or, equivalently, $\rk (N)=\rk(\im \tilde{P}')$.
\end{itemize}
These conditions can effectively be checked by computing a free-basis for $\Fix \phi_1\cap \cdots \cap \Fix \phi_k$ with Theorem~\ref{maslakova} and pull-backs of graphs, and then computing the ranks $\rk(\im \tilde{P}')$ and $\rk(N)$ with basic linear algebra techniques. So, we can effectively decide whether $\Fix \Psi_1 \cap \cdots \cap \Fix \Psi_k$ is finitely generated or not.

Finally, let us assume it is so, and let us compute a basis for $\Fix \Psi_1 \cap \cdots \cap \Fix \Psi_k$.

If we are in the situation (i) then $\Fix \phi_1 \cap \cdots \cap \Fix \phi_k =\langle u\rangle$, $u\rho \neq 0$, and $M\cap \im \tilde{P}'=N=\{0\}$ so, the only elements in $\Fix \Psi_1 \cap \cdots \cap \Fix \Psi_k$ are those of the form $t^a u^r$ with $a(I_m-\tilde{Q})=r\cdot u\rho \tilde{P}=0$. That is, $\Fix \Psi_1 \cap \cdots \cap \Fix \Psi_k =\langle u,\, t^{d_1},\ldots , t^{d_s}\rangle$ where $\langle d_1,\ldots ,d_s\rangle =E_1(Q_1)\cap \cdots \cap E_1(Q_k)\leqslant \Z^m$.

If we are in situation (ii), then we can compute a set $\{c_1, \ldots ,c_q\}\subset \Z^n$ of coset representatives of $N\tilde{P}^{\prime -1}$ in $\im\rho'$, namely $\im\rho' =(N\tilde{P}^{\prime -1})c_1 \sqcup \cdots \sqcup (N\tilde{P}^{\prime -1})c_q$. Having computed a free-basis $\{v_1, \ldots, v_p\}$ for $\Fix \phi_1 \cap \cdots \cap \Fix \phi_k$, we can choose arbitrary preimages $y_1,\ldots ,y_q$ of $c_1,\ldots ,c_q$ up in $\Fix \phi_1 \cap \cdots \cap \Fix \phi_k$, and we get a set of right coset representatives of $(\Fix \Psi_1  \cap \cdots \cap \Fix \Psi_k)\pi =N\tilde{P}^{\prime -1} \rho^{\prime -1}$ in $\Fix \phi_1 \cap \cdots \cap \Fix \phi_k$,
 \begin{equation}\label{equ}
\Fix \phi_1 \cap \cdots \cap \Fix \phi_k =(N\tilde{P}^{\prime -1}\rho^{\prime -1})y_1 \sqcup \cdots \sqcup (N\tilde{P}^{\prime -1}\rho^{\prime -1})y_q.
 \end{equation}
Now, we build the Schreier graph for $N\tilde{P}^{\prime -1}\rho^{\prime -1} \leqslant_{fi} \Fix \phi_1 \cap \cdots \cap \Fix \phi_k$ with respect to $\{v_1, \ldots, v_p \}$ in the following way: (1) take the cosets from~\eqref{equ} as vertices, and with no edge; (2) for every vertex $(N\tilde{P}^{\prime -1}\rho^{\prime -1})y_i$ and every letter $v_j$,  add an edge labeled $v_j$ from $(N\tilde{P}^{\prime -1}\rho^{\prime -1})y_i$ to $(N\tilde{P}^{\prime -1}\rho^{\prime -1})y_iv_j$, algorithmically identified among the available vertices by repeatedly solving the membership problem for $N\tilde{P}^{\prime -1}\rho^{\prime -1}$ (note that we can easily do this by abelianizing the candidate and checking whether it belongs to $N\tilde{P}^{\prime -1}$). Once we have run over all $i=1,\ldots ,q$ and all $j=1,\ldots ,p$, we have computed the full (and finite!) Schreier graph, from which we can select a maximal tree and obtain a free-basis $\{u_1,\ldots ,u_r\}$ for the subgroup corresponding to closed paths at the basepoint, i.e., for $N\tilde{P}^{\prime -1} \rho^{\prime -1}= (\Fix \Psi_1 \cap \cdots \cap \Fix \Psi_k)\pi$. Finally, solving linear systems of equations (which must be mandatorily compatible), we obtain vectors $e_1,\ldots ,e_r\in \Z^m$ such that $t^{e_1}u_1, \ldots ,t^{e_r}u_r \in \Fix\Psi_1 \cap \cdots \cap \Fix \Psi_k$. We conclude that $\{t^{e_1}u_1, \ldots ,t^{e_r}u_r,\, t^{d_1},\ldots ,t^{d_s}\}$ is a basis for $\Fix\Psi_1 \cap \cdots \cap \Fix \Psi_k$.
\end{proof}

\begin{proof}[Proof of Theorem~\ref{closure}]
From the given generators, compute a basis for $H$, say $\{t^{a_1}u_1, \ldots, t^{a_r}u_r,$ $t^{b_1}, \ldots, t^{b_s}\}$. Now, using Theorem~\ref{main-1}, we can compute automorphisms $\Psi_1,\ldots ,\Psi_k\in \Aut(G)$ such that $\Aut_H(G)=\langle \Psi_1,\ldots ,\Psi_k\rangle$. So, we have that $\acl_G(H)=\Fix\Psi_1\cap \cdots \cap \Fix\Psi_k$. Finally, using Proposition~\ref{com-int}, we can decide whether this intersection is finitely generated or not and, in the affirmative case, compute a basis for it.
\end{proof}

\begin{proof}[Proof of Corollary~\ref{deciding}]
Given generators for $H\leqslant_{fg} G$, apply Theorem~\ref{closure}. If $\acl_G(H)$ is not finitely generated then conclude that $H$ is not auto-fixed. Otherwise, we get a set of automorphisms $\Psi_1, \ldots, \Psi_k \in \Aut(G)$ such that $\acl_{G}(H)= \Fix \Psi_1 \cap \cdots \cap \Fix \Psi_k$, and a basis for $\acl_G(H)\geqslant H$. Now $H$ is auto-fixed if and only if this last inclusion is an equality (which can be algorithmically checked by using a solution to the membership problem in $G$; see~\cite[Prop.~1.11]{DV}); and in this case, $\Psi_1, \ldots, \Psi_k$ are the automorphisms such that $H=\Fix \Psi_1 \cap \cdots \cap \Fix \Psi_k$.
\end{proof}

\section{Appendix: computation of relations}\label{rels}\label{6}

Let us go back to the details of the proof of Theorem~\ref{main-1} and complete it by computing a finite set of defining relations for $\Aut_H(G)$.

\begin{proof}[Proof of Theorem~\ref{main-1} continued (relations part)]
We have already computed a finite set of generators for $\Aut_H(G)$. To find the defining relations, we distinguish again the cases $r=0$, $r\geqslant 2$, and $r=1$ (in increasing order of difficulty):

\noindent $\bullet~$ \emph{Case 1: $r=0$}. Here, we have $H=L_H$, and we know that $\Aut_H(G)$ is (finitely) generated by the automorphisms of $G$ of the form (1) $\Psi_{\phi, I_m, 0}$, with $\phi$ running over the Nielsen automorphisms of $F_n$; (2) $\Psi_{id, Q, 0}$, with $Q$ running over the generators of $\Aut_{L_H}(\Z^m)$; and (3) $\Psi_{id, I_m, 1_{i,j}}$, with $i=1,\ldots ,n$, $j=1,\ldots ,m$. Therefore, from~\cite[Thm.~5.5]{DV}, we deduce that $\Aut_H(G)\simeq \Mat_{n\times m} \rtimes \big( \Aut_{L_H}(\Z^m)\times \Aut(F_n)\big)$ with the natural action. Hence, we can easily compute an explicit finite presentation for this group by using the presentation for $\Aut_{L_H}(\Z^m)$ we got from Day's Theorem~\ref{Day}, any know presentation for $\Aut(F_n)$ (see, for example, \cite{AFV}), and the standard presentation for $\Mat_{n\times m}\simeq \Z^{nm}$.

\noindent $\bullet~$ \emph{Case 2: $r\geqslant 2$}. In this case, we already know that $\Aut_H(G)=\langle \Psi'_1, \ldots ,\Psi'_{\ell} \rangle$. Let us find a complete set of defining relations for this set of generators.

Observe first that, for every $\Psi\in \Aut_W(G)$, the decomposition $\Psi=\Psi'\Gamma_{x}$ mentioned in~\eqref{=} is unique: if $\Psi'\Gamma_x =\Psi''\Gamma_y$, with $\Psi',\Psi''\in \Aut_H(G)$ and $x,y\in F_n$, then $x^{-1}u_1 x=y^{-1}u_1 y$ and $x^{-1}u_2 x=y^{-1}u_2 y$, which implies that $xy^{-1}$ commutes with the freely independent elements $u_1, u_2$ and so, $xy^{-1}=1$; hence, $\Gamma_x=\Gamma_y$ and $\Psi'=\Psi''$. In other words, $\Aut_H(G)\cap \Inn(G)=\{\id_G\}$ and so,
 $$
\Aut_W(G)/\Inn(G)=\Aut_H(G)\Inn(G)/\Inn(G) \simeq \Aut_H(G)/\big( \Aut_H(G)\cap \Inn(G)\big) =\Aut_H(G).
 $$

We have the following two sources of natural relations among the $\Psi'_i$'s. From~\eqref{pres}, for each $i=1,\ldots ,d$ we have $\id_G=R_i(\Psi_1, \ldots, \Psi_{\ell})=R_i(\Psi'_1 \Gamma_{x_1}, \ldots, \Psi'_{\ell} \Gamma_{x_{\ell}})=R_i(\Psi'_1, \ldots, \Psi'_{\ell})\Gamma_{y_i} =R_i(\Psi'_1, \ldots, \Psi'_{\ell})$, where $y_i\in F_n$ must be 1, again, because $r\geqslant 2$. On the other hand, for each one of the $n$ generating letters of $F_n$, say $z_1,\ldots ,z_n$, compute an expression for the conjugation $\Gamma_{z_j} \in \Inn(G)\leqslant \Aut_W(G)$ in terms of $\Psi_1, \ldots ,\Psi_{\ell}$, say $\Gamma_{z_j}=S_j(\Psi_1, \ldots ,\Psi_{\ell})$, and we have $\Gamma_{z_j}=S_j(\Psi_1, \ldots ,\Psi_{\ell})=S_j(\Psi'_1\Gamma_{x_1}, \ldots ,\Psi'_{\ell}\Gamma_{x_{\ell}})=S_j(\Psi'_1, \ldots ,\Psi'_{\ell})\Gamma_{y_j}$ for some $y_j\in F_n$; but then $\id_G=S_j(\Psi'_1, \ldots ,\Psi'_{\ell})\Gamma_{y_jz_j^{-1}}=S_j(\Psi'_1, \ldots ,\Psi'_{\ell})$, $j=1,\ldots ,n$, gives us a second set of relations for $\Aut_H(G)$ (here, again, $y_jz_j^{-1}=1$ since $r\geqslant 2$). Therefore,
 $$
\begin{array}{rcl}
\Aut_H(G) & = & \Aut_W(G)/\Inn(G) \\ & = & \langle \Psi_1, \ldots ,\Psi_{\ell} \mid R_1, \ldots ,R_d\rangle /\Inn(G) \\ & = & \langle \Psi'_1, \ldots ,\Psi'_{\ell} \mid R_1, \ldots ,R_d, S_1,\ldots ,S_n\rangle.
\end{array}
 $$
(Note that $w(\Psi_1, \ldots ,\Psi_{\ell}) \mapsto w(\Psi'_1, \ldots ,\Psi'_{\ell})$ or, equivalently, $\Psi \mapsto \Psi'=\Psi\Gamma_{x^{-1}}$ for the unique possible $x\in F_n$, is the canonical projection $\Aut_W(G) \twoheadrightarrow \Aut_H(G)\simeq \Aut_W(G)/\Inn(G)$.)

\noindent $\bullet~$ \emph{Case 3: $r=1$}. Here, $H=\langle t^a u,\, t^{b_1},\ldots ,t^{b_s}\rangle\leqslant G$ with $1\neq u\in F_n$ (for notational simplicity, we have deleted the subindex 1 from $u$ and $a$). This case is a bit more complicated than Case~2 because the decomposition $\Psi=\Psi'\Gamma_x$ from~\eqref{=} is not unique now; additionally, $\Aut_H(G)$ contains some non-trivial conjugation, namely $\Gamma_{\hat{u}}$, and so we cannot mod out $\Inn(G)$ from $\Aut_W(G)$ because this would kill part of $\Aut_H(G)$.

In the present case, we know that $\Aut_H(G)=\langle \Psi'_1,\ldots ,\Psi'_{\ell}, \Gamma_{\hat{u}}\rangle$. Let us adapt the two previous sources of natural relations among them, and discover a third one. From~\eqref{pres}, for each $i=1,\ldots ,d$ we have $\id_G=R_i(\Psi_1, \ldots, \Psi_{\ell})=R_i(\Psi'_1 \Gamma_{x_1}, \ldots, \Psi'_{\ell} \Gamma_{x_{\ell}})=R_i(\Psi'_1, \ldots, \Psi'_{\ell})\Gamma_{y_i}$, for some $y_i\in F_n$. But both $\id_G$ and $R_i(\Psi'_1, \ldots, \Psi'_{\ell})$ fix $t^a u$ so, $y_i$ must equal $\hat{u}^{\alpha_i}$ for some $\alpha_i\in \Z$. Therefore, $\id_G=R_i(\Psi'_1, \ldots, \Psi'_{\ell})\Gamma_{\hat{u}}^{\alpha_i}$, $i=1,\ldots ,d$, is a first set of relations for $\Aut_H(G)$.

On the other hand, for each generating letter, $z_j$, of $F_n$, $j=1,\ldots ,n$, we have the equality $\Gamma_{z_j}=S_j(\Psi_1, \ldots ,\Psi_{\ell})=S_j(\Psi'_1\Gamma_{x_1}, \ldots ,\Psi'_{\ell}\Gamma_{x_{\ell}})=S_j(\Psi'_1, \ldots ,\Psi'_{\ell})\Gamma_{y_j}$, for some $y_j\in F_n$. But then $\id_G=S_j(\Psi'_1, \ldots ,\Psi'_{\ell})\Gamma_{y_jz_j^{-1}}$, which implies $y_jz_j^{-1}=\hat{u}^{\beta_j}$ for some $\beta_j\in \Z$. Therefore, $\id_G =S_j(\Psi'_1, \ldots ,\Psi'_{\ell})\Gamma_{\hat{u}}^{\beta_j}$, $j=1,\ldots ,n$, is a second set of relations for $\Aut_H(G)$.

Finally, observe that for $k=1,\ldots ,\ell$, $\hat{u}\Psi'_k=t^{c_k}\hat{u}$ for some $c_k\in \Z^m$ and thus, $\Gamma_{\hat{u}}$ commutes with $\Psi'_k$. Therefore, $\Psi'_k\Gamma_{\hat{u}}=\Gamma_{\hat{u}}\Psi'_k$, $k=1,\ldots ,\ell$, is a third set of relations for $\Aut_H(G)$.

We are going to prove that
 \begin{equation}\label{isom}
\Aut_H(G) \simeq \left\langle \Psi'_1, \ldots ,\Psi'_{\ell}, \Gamma_{\hat{u}} \,\, \Big| \begin{array}{ccccc} R_i(\Psi'_1, \ldots, \Psi'_{\ell})\Gamma_{\hat{u}}^{\alpha_i}, & & S_j(\Psi'_1, \ldots ,\Psi'_{\ell})\Gamma_{\hat{u}}^{\beta_j}, & & \Psi'_k\Gamma_{\hat{u}}=\Gamma_{\hat{u}}\Psi'_k \\ _{i=1,\ldots ,d} & & _{j=1,\ldots ,n} & & _{k=1,\ldots ,\ell} \end{array} \right\rangle.
 \end{equation}

To this goal, denote by ${\mathcal G}$ the group presented by the presentation on the right hand side, where elements are formal words on the `symbols' $\{\Psi'_1, \ldots ,\Psi'_{\ell}, \Gamma_{\hat{u}} \}$ subject to the relations indicated (we abuse notation, denoting by $\Psi'_1,\ldots ,\Psi'_{\ell}, \Gamma_{\hat{u}}$ both the corresponding symbols in ${\mathcal G}$, and the corresponding automorphisms in $\Aut_H(G)$, the real meaning being always clear from the context). Let us construct a map $f\colon \Aut_H(G)\to {\mathcal G}$, and a group homomorphism $G\leftarrow {\mathcal G} :\!g$ such that $fg=\id_{\Aut_H(G)}$ and $gf=\id_{\mathcal G}$. This will suffice to prove~\eqref{isom} and finish the argument.

Define $g$ by sending the symbol $\Psi'_k$ to the automorphism $\Psi'_k$, $k=1,\ldots ,\ell$, and the symbol $\Gamma_{\hat{u}}$ to the automorphism $\Gamma_{\hat{u}}$; since, as we have proved in the three previous paragraphs, the relations from ${\mathcal G}$ are really satisfied in $\Aut_H(G)$, $g$ determines a well defined homomorphism from ${\mathcal G}$ to $\Aut_H(G)$. (For later use, we emphasize the meaning of this: every equality holding symbolically in ${\mathcal G}$ holds also genuinely in $\Aut_H(G)$.) On the other hand, for $\Psi\in \Aut_H(G)$, define $\Psi f\in {\mathcal G}$ as follows: write $\Psi\in \Aut_H(G)\leqslant \Aut_W(G)$ as a word on ${\Psi_1, \ldots ,\Psi_{\ell}}$, say $\Psi=v(\Psi_1, \ldots ,\Psi_{\ell})$, compute $\Psi=v(\Psi_1, \ldots ,\Psi_{\ell})=v(\Psi'_1\Gamma_{x_1}, \ldots ,\Psi_{\ell}\Gamma_{x_{\ell}})=v(\Psi'_1, \ldots ,\Psi'_{\ell})\Gamma_y =v(\Psi'_1, \ldots ,\Psi'_{\ell})\Gamma_{\hat{u}}^{\rho}$ (in $\Aut_H(G)$ !), where $y=\hat{u}^{\rho}$ for some $\rho\in \Z$ since both $\Psi$ and $v(\Psi'_1, \ldots ,\Psi'_{\ell})$ fix $t^a u$; and, finally, define $\Psi f$ to be the word $v(\Psi'_1, \ldots ,\Psi'_{\ell})\Gamma_{\hat{u}}^{\rho}\in {\mathcal G}$.

First, we have to see that $f$ is well defined. That is, take $\Psi=w(\Psi_1,\ldots ,\Psi_{\ell})$ another expression for $\Psi$, write $\Psi=w(\Psi_1, \ldots ,\Psi_{\ell})=w(\Psi'_1, \ldots ,\Psi'_{\ell})\Gamma_{\hat{u}}^{\tau}$ (in $\Aut_H(G)$ !) for the appropriate integer $\tau\in \Z$, and we have to prove that the equality $v(\Psi'_1, \ldots ,\Psi'_{\ell})\Gamma_{\hat{u}}^{\rho}=w(\Psi'_1, \ldots ,\Psi'_{\ell})\Gamma_{\hat{u}}^{\tau}$ holds, abstractly, in ${\mathcal G}$. From the fact $v(\Psi_1,\ldots ,\Psi_{\ell})=\Psi=w(\Psi_1, \ldots ,\Psi_{\ell})$ (equalities happening in the group~\eqref{pres}), we deduce that the word $v(\Psi_1,\ldots ,\Psi_{\ell})^{-1} w(\Psi_1, \ldots ,\Psi_{\ell})$ is formally a product of conjugates of $R_1(\Psi_1,\ldots ,\Psi_{\ell}), \ldots ,R_d(\Psi_1,\ldots ,\Psi_{\ell})$, say
 $$
v(\Psi_1,\ldots ,\Psi_{\ell})^{-1} w(\Psi_1, \ldots ,\Psi_{\ell}) =\prod_{k=1}^N \big( R_{i_k}^{\epsilon_k}(\Psi_1,\ldots ,\Psi_{\ell})\big)^{c_k(\Psi_1,\ldots ,\Psi_{\ell})}.
 $$
Particularizing this identity on $\Psi'_1, \ldots ,\Psi'_{\ell}\in {\mathcal G}$, and working in ${\mathcal G}$ (i.e., only using symbolically the defining relations for ${\mathcal G}$), we have that
 $$
v(\Psi'_1, \ldots ,\Psi'_{\ell})^{-1}w(\Psi'_1, \ldots ,\Psi'_{\ell})= \prod_{k=1}^N \big( R_{i_k}^{\epsilon_k}(\Psi'_1,\ldots ,\Psi'_{\ell})\big)^{c_k(\Psi'_1,\ldots ,\Psi'_{\ell})}=
 $$
 $$
=\prod_{k=1}^N \big( \Gamma_{\hat{u}}^{-\epsilon_k \alpha_{i_k}}\big)^{c_k(\Psi'_1,\ldots ,\Psi'_{\ell})} = \prod_{k=1}^N \Gamma_{\hat{u}}^{-\epsilon_k \alpha_{i_k}} =\Gamma_{\hat{u}}^{-\sum_{k=1}^N \epsilon_k \alpha_{i_k}}.
 $$
But, applying $g$ (i.e., reading the above equality in $\Aut_H(G)$), we have
 $$
\id_G =v(\Psi_1,\ldots ,\Psi_{\ell})^{-1} w(\Psi_1, \ldots ,\Psi_{\ell})=\Gamma_{\hat{u}}^{-\rho} v(\Psi'_1,\ldots ,\Psi'_{\ell})^{-1} w(\Psi'_1, \ldots ,\Psi'_{\ell})\Gamma_{\hat{u}}^{\tau}=\Gamma_{\hat{u}}^{\tau-\rho-\sum_{k=1}^N \epsilon_k \alpha_{i_k}}
 $$
and so, the exponent must be zero, $\tau-\rho-\sum_{k=1}^N \epsilon_k \alpha_{i_k}=0$, because $n\geqslant 2$. Going back to ${\mathcal G}$, we conclude that $\Gamma_{\hat{u}}^{-\rho} v(\Psi'_1, \ldots ,\Psi'_{\ell})^{-1}w(\Psi'_1, \ldots ,\Psi'_{\ell})\Gamma_{\hat{u}}^{\tau}=\Gamma_{\hat{u}}^{\tau-\rho-\sum_{k=1}^N \epsilon_k \alpha_{i_k}}=1$, showing that the map $f$ is well defined.

Now consider the composition $fg\colon \Aut_H(G)\to {\mathcal G}\to \Aut_H(G)$: for every $\Psi\in \Aut_H(G)$, write (in $\Aut_H(G)$ !) $\Psi=v(\Psi_1, \ldots ,\Psi_{\ell})=v(\Psi'_1, \ldots ,\Psi'_{\ell})\Gamma_{\hat{u}}^{\rho}$, $\rho\in \Z$, and we have $\Psi f=v(\Psi'_1, \ldots ,\Psi'_{\ell})\Gamma_{\hat{u}}^{\rho}\in {\mathcal G}$. But then, $\Psi fg=\big( v(\Psi'_1, \ldots ,\Psi'_{\ell})\Gamma_{\hat{u}}^{\rho}\big) g=v(\Psi'_1, \ldots ,\Psi'_{\ell})\Gamma_{\hat{u}}^{\rho}=\Psi$ (in $\Aut_H(G)$ !). Hence, $fg=\id_{\Aut_H(G)}$.

Finally, consider the composition $gf\colon {\mathcal G}\to \Aut_H(G)\to {\mathcal G}$. Take $k=1,\ldots ,\ell$ and, in order to compute $\Psi'_k gf=\Psi'_k f$, we have to express $\Psi'_k\in \Aut_H(G)$ as a word on $\Psi_1, \ldots ,\Psi_{\ell}$; take, for example, $\Psi'_k=\Psi_k \Gamma_{x_k}^{-1}=\Psi_k \Gamma_{x_k(z_1, \ldots ,z_n)}^{-1}=\Psi_k x_k (\Gamma_{z_1}, \ldots ,\Gamma_{z_n})^{-1}=\Psi_k x_k (S_1(\Psi_1, \ldots ,\Psi_{\ell}), \ldots ,S_n(\Psi_1, \ldots ,\Psi_{\ell}))^{-1}$; then, rewrite in terms of $\Psi'_1, \ldots ,\Psi'_{\ell}$,
 $$
\Psi'_k=\Psi_k x_k (S_1(\Psi_1, \ldots ,\Psi_{\ell}), \ldots ,S_n(\Psi_1, \ldots ,\Psi_{\ell}))^{-1}=
\Psi'_k x_k (S_1(\Psi'_1, \ldots ,\Psi'_{\ell}), \ldots ,S_n(\Psi'_1, \ldots ,\Psi'_{\ell}))^{-1}\Gamma_{\hat{u}}^{\rho},
 $$
for the appropriate integer $\rho\in \Z$; and we have, in ${\mathcal G}$ (i.e., only using symbolically the defining relations for ${\mathcal G}$),
 $$
\begin{array}{rcl}
\Psi'_k gf =\Psi'_k f & = & \Psi'_k x_k (S_1(\Psi'_1, \ldots ,\Psi'_{\ell}), \ldots ,S_n(\Psi'_1, \ldots ,\Psi'_{\ell}))^{-1}\Gamma_{\hat{u}}^{\rho} \\ & = & \Psi'_k x_k (\Gamma_{\hat{u}}^{-\beta_1}, \ldots ,\Gamma_{\hat{u}}^{-\beta_n})^{-1}\Gamma_{\hat{u}}^{\rho} \\ & = & \Psi'_k \Gamma_{\hat{u}}^{\, x_k^{\rm ab}\beta^T} \Gamma_{\hat{u}}^{\rho} \\ & = & \Psi'_k \Gamma_{\hat{u}}^{\, x_k^{\rm ab}\beta^T+\rho},
\end{array}
 $$
where $\beta =(\beta_1, \ldots ,\beta_n)\in \Z^n$. But, applying $g$, using $fg=\id_{\Aut_H(G)}$, and cancelling $\Psi'_i$ from the left, we obtain $\id_G =\Gamma_{\hat{u}}^{\, x_k^{\rm ab}\beta^T+\rho}$ and so, $x_k^{\rm ab}\beta^T+\rho =0$. Hence, back in ${\mathcal G}$, $\Psi'_k gf=\Psi'_k$, for $k=1,\ldots ,\ell$.

Similarly,
 $$
\begin{array}{rcl}
\Gamma_{\hat{u}} gf =\Gamma_{\hat{u}} f & = & \big(\hat{u}(\Gamma_{z_1},\ldots ,\Gamma_{z_n}) \big) f \\ & = & \big(\hat{u}(S_1(\Psi_1, \ldots ,\Psi_{\ell}),\ldots ,S_n(\Psi_1, \ldots ,\Psi_{\ell})) \big) f \\ & = & \hat{u}(S_1(\Psi'_1, \ldots ,\Psi'_{\ell}),\ldots ,S_n(\Psi'_1, \ldots ,\Psi'_{\ell})) \Gamma_{\hat{u}}^{\rho} \\ & = & \hat{u}(\Gamma_{\hat{u}}^{-\beta_1}, \ldots ,\Gamma_{\hat{u}}^{-\beta_n})\Gamma_{\hat{u}}^{\rho} \\ & = & \Gamma_{\hat{u}}^{\, -\hat{u}^{\rm ab}\beta^T+\rho},
\end{array}
 $$
for the appropriate integer $\rho\in Z$. But, applying $g$, we obtain $\Gamma_{\hat{u}}=\Gamma_{\hat{u}}^{\, -\hat{u}^{\rm ab}\beta^T+\rho}$ (in $\Aut_H(G)$ !) and so, $-\hat{u}^{\rm ab}\beta^T+\rho=1$. Hence, back in ${\mathcal G}$, $\Gamma_{\hat{u}}gf=\Gamma_{\hat{u}}$, finishing the proof that $gf=\id_{\mathcal G}$.

This completes the proof of the isomorphism~\eqref{isom} and so, the proof of the Theorem.
\end{proof}

The above proof that stabilizers of subgroups of $G=\Z^m\times F_n$ are finitely presented (and a finite presentation is computable) makes a strong use of the fact that the center of $G$ is $\Z^m$, i.e., the elements of the form $t^a$ commute with everybody in $G$. For this reason, this proof is far from generalizing to arbitrary right angled Artin groups, providing an analog of Day's Theorem~\ref{Day} for real elements instead of conjugacy classes. This suggests the following question, which is open as far as we know.

\begin{que}
Is it true that, for every finitely generated subgroup of a right angled Artin group, $H\leqslant_{fg} A(\Gamma)$, the stabilizer $\Aut_H(A(\Gamma))$ is finitely generated ? and finitely presented ? and a presentation algorithmically computable from the given generators for $H$ ?
\end{que}

\section*{Acknowledgements}

\noindent The authors acknowledge partial support from the Spanish Agencia Estatal de Investigaci\'on, through grant MTM2017-82740-P (AEI/FEDER, UE), and also from the Barcelona Graduate School of Mathematics through the ``Mar\'{\i}a de Maeztu" Programme for Units of Excellence in R\&D (MDM-2014-0445). The first named author wants to thank the hospitality and support of the Barcelona Graduate School of Mathematics and the Universitat Polit\`ecnica de Catalunya.


\begin{thebibliography}{99}

\bibitem{AFV} H. Armstrong, B. Forrest, and K. Vogtmann, ``A presentation for $\Aut(F_n)$'', \emph{J. Group Theory} \textbf{11}(2) (2008), 267--276.

\bibitem{BH} M. Bestvina, M. Handel, ``Train tracks and automorphisms of free groups'', \emph{Ann. of Math.} \textbf{135} (1992), 1--51.

\bibitem{BMMV} O. Bogopolski, A. Martino, O. Maslakova, and E. Ventura, ``Free-by-cyclic groups have solvable conjugacy problem'', \emph{Bulletin of the London Mathematical Society} \textbf{38(5)} (2006), 787--794.

\bibitem{BM} O. Bogopolski, O. Maslakova, ``An algorithm for finding a basis of the fixed point subgroup of an automorphism of a free group'', \emph{Internat. J. Algebra Comput.} \textbf{26}(1) (2016), 29--67.

\bibitem{BV} O. Bogopolski, E. Ventura, ``On endomorphisms of torsion-free hyperbolic groups'', \emph{International Journal of Algebra and Computation} \textbf{21}(8) (2011), 1415--1446.

\bibitem{Cu} M. Culler, ``Finite groups of outer automorphisms of a free group'', Contributions to group theory, 197--207, \emph{Contemp. Math.} \textbf{33}, Amer. Math. Soc., Providence, RI, 1984.

\bibitem{Day} M. Day, ``Full-featured peak reduction in right-angled Artin groups'', \emph{Algebr. Geom. Topol.} \textbf{14}(3) (2014), 1677--1743.

\bibitem{DV} J. Delgado, E. Ventura, ``Algorithmic problems for free-abelian times free groups'', \emph{Journal of Algebra} \textbf{391} (2013), 256--283.

\bibitem{Dummit-Foote} D. Dummit, R. Foote, ``Abstract Algebra", Prentice Hall, Englewood Cliffs, N.J., 1991.

\bibitem{DS} J.L. Dyer, G.P. Scott, ``Periodic automorphisms of free groups'', \emph{Comm. Alg.} \textbf{3} (1975), 195--201.

\bibitem{FH} M. Feighn, M. Handel, ``Algorithmic constructions of relative train track maps and CT's'', \emph{Groups Geom. Dyn.} \textbf{12}(3) (2018), 1159--1238.

\bibitem{IT} W. Imrich and E.C. Turner, ``Endomorphisms of free groups and their fixed points'', \emph{Math. Proc. Cambridge Philos. Soc.} \textbf{105} (1989), 421--422.

\bibitem{LS} R.C. Lyndon, P. Schupp, ``Combinatorial Group Theory'', reprint ed. Springer, Mar. 2001.

\bibitem{McCool} J. McCool, ``Some finitely presented subgroups of the automorphism group of a free group'', \emph{Journal of algebra} \textbf{35}(1--3) (1975), 205--213.

\bibitem{MV} A. Martino, E. Ventura, ``On automorphism-fixed subgroups of a free group'', \emph{Journal of Algebra} \textbf{230} (2000), 596--607.

\bibitem{MV2} A. Martino, E. Ventura, ``Fixed subgroups are compressed in free groups'', \emph{Comm. in Algebra} \textbf{32}(10) (2004), 3921--3935.

\bibitem{MV3} A. Martino, E. Ventura, ``Examples of retracts in free groups that are not the fixed subgroup of any automorphism'', \emph{Journal of Algebra} \textbf{269} (2003), 735--747.

\bibitem{Ma} O. Maslakova, ``The fixed point group of a free group automorphism'', \emph{Algebra i Logika} \textbf{42}(4) (2003), 422--472. Translated to English at \emph{Algebra and Logic} \textbf{42}(4) (2003), 237--265.

\bibitem{St} J.R. Stallings, ``Finiteness properties of matrix representations'', \emph{Annals of Mathematics} \textbf{124} (1986), 337--346.

\bibitem{V} E. Ventura, ``Computing fixed closures in free groups'', \emph{Illinois Journal of Mathematics} \textbf{54}(1) (2011), 175--186.

\bibitem{W} J.H.C. Whitehead, ``On equivalent sets of elements in a free group'', \emph{Annals of Mathematics} \textbf{37} (1936) 782--800.

\end{thebibliography}
\end{document}